\documentclass{article}
\usepackage{amsmath}
\usepackage{amssymb, amscd, amsthm, amsfonts, latexsym}
\usepackage{enumerate}
\usepackage[dvips]{graphicx,color}
\usepackage{fancybox}
\usepackage{ascmac}
\usepackage{wrapfig,amsmath}
\usepackage{ulem}
\usepackage{leftidx}

\usepackage{latexsym}
\newcommand{\no}{\noindent}
\newtheorem{thm}{Theorem}[section]
\newtheorem{prop}[thm]{Proposition}
\newtheorem{cor}[thm]{Corollary}
\newtheorem{rem}[thm]{Remark}

\newtheorem{lem}[thm]{Lemma}

\newtheorem{notation}[thm]{Notation}
\numberwithin{equation}{section}

\newcommand{\R}{\mathbb{R}}

\newcommand{\N}{\mathbb{N}}
\newcommand{\rN}{\mathrm{N}}
\newcommand{\ds}{\displaystyle}

\newcommand{\sm}{\setminus}
\newcommand{\pd}{\partial}
\newcommand{\al}{\alpha}

\newcommand{\ga}{\gamma}
\newcommand{\Ga}{\Gamma}

\newcommand{\De}{\Delta}

\newcommand{\f}{\varphi}
\newcommand{\ome}{\omega}
\newcommand{\Ome}{\Omega}

\renewcommand{\(}{\left(}
\renewcommand{\)}{\right)}
\renewcommand{\lvert}{\left\vert}
\renewcommand{\rvert}{\right\vert}
\DeclareMathOperator*{\esssup}{ess\,sup}

\DeclareMathOperator{\conv}{conv}
\DeclareMathOperator{\supp}{supp}

\DeclareMathOperator{\dist}{dist}

\DeclareMathOperator{\Pos}{Pos}

\def\vect#1{\mbox{\boldmath $#1$}}

\begin{document}

\title{\bf Movement of time-delayed hot spots in Euclidean space for special initial states}

\author{\bf Shigehiro Sakata and Yuta Wakasugi}

\date{\today}

\maketitle

\begin{abstract} 
We consider the Cauchy problem for the damped wave equation under the initial state that the sum of an initial position and an initial velocity vanishes. 
When the initial position is non-zero, non-negative and compactly supported, we study the large time behavior of the spatial null, critical, maximum and minimum sets of the solution. 

The spatial null set becomes a smooth hyper-surface homeomorphic to a sphere after a large enough time.
The spatial critical set has at least three points after a large enough time.
The set of spatial maximum points escapes from the convex hull of the support of the initial position. 
The set of spatial minimum points consists of one point after a large time, and the unique spatial minimum point converges to
the centroid of the initial position at time infinity.\\

\no{\it Keywords and phrases}.
Damped wave equation,
diffusion phenomenon, time-delayed hot spots, Nishihara decomposition\\
\no 2010 {\it Mathematics Subject Classification}:
35L15, 35B38, 35C15, 35B40, 35K05.
\end{abstract}

\section{Introduction}
Let $f$ and $g$ be real-valued smooth functions defined on $\R^n$.
We consider the {\it damped wave equation} with initial data $(f,g)$,
\begin{equation}\label{DWfg}
\begin{cases}
\ds \( \frac{\pd^2}{\pd t^2} -\De + \frac{\pd}{\pd t} \) u(x,t) =0, &x \in \R^n ,\ t>0,\\
\ds \( u, \frac{\pd u}{\pd t} \) (x,0)=(f, g)(x), &x \in \R^n .
\end{cases}
\end{equation}

The damped wave equation describes several phenomena. 
One is the propagation of the wave with friction. 
Another is the diffusion of the heat with finite propagation speed. 
In this paper, we study the equation \eqref{DWfg} in the latter sense. 
We refer to [L] for the deriving process as the heat equation with finite propagation speed (see also [SY, Introduction]).

It is one of the most fundamental properties of the damped wave equation that, as $t$ goes to infinity, the unique classical solution of \eqref{DWfg} approaches to that of the corresponding heat equation,
\begin{equation}
P_n (t) (f+g) (x) = \frac{1}{( 4\pi t)^{n/2}} \int_{\R^n} \exp \( -\frac{r^2}{4t} \) (f+g) (y) dy ,\ x \in \R^n ,\ t>0 ,\ r=\lvert x-y \rvert , 
\end{equation}
if $f+g$ does not vanish. 
Such a property is called the {\it diffusion phenomenon} and has been studied by several researchers [FG, HO, M, MN, Nar, Nis, SY, YM].
We also refer to [I] for the diffusion phenomenon in an exterior domain.

In order to compare the damped wave and heat equations more deeply, in [SY], the authors investigated the shape of the solution of \eqref{DWfg} when $f$ and $g$ are compactly supported, and $f+g$ is non-zero and non-negative. 
In particular, the authors studied the spatial maximum set of the solution of \eqref{DWfg},
\begin{equation}
\mathcal{H} (t) = \left\{ x \in \R^n \lvert u(x,t) = \max u ( \cdot ,t)  \right\} \right. ,\ t>0 .
\end{equation}
Spatial maximum points of the solution of the (usual) heat equation are called {\it hot spots}. 
Since \eqref{DWfg} is the heat equation with finite propagation speed, we call spatial maximum points of the solution of \eqref{DWfg} {\it time-delayed hot spots}. 
The authors showed that $\mathcal{H} (t)$ has the following properties:
\begin{enumerate}[(1)]
\item After a large time, $\mathcal{H} (t)$ is contained in the convex hull of the support of $f+g$. 
\item After a large time, $\mathcal{H} (t)$ consists of one point.
\item As $t$ goes to infinity, the unique time-delayed hot spot converges to the centroid of $h$.
\end{enumerate}
The first and last properties correspond to the results of hot spots shown in [CK, Theorem 1]. 
The second property corresponds to the fact of hot spots indicated in [JS, Introduction]. 
In other words, the large time behavior of time-delayed hot spots is same as that of hot spots. 
Furthermore, the authors gave examples of $(f,g)$ such that $\mathcal{H} (t)$ is contained in the exterior of the convex hull of $f+g$ for some small time, which claims that the short time behavior of time-delayed hot spots is not similar to that of hot spots.

This paper is a continuation of the study in [SY]. 
We investigate the shape of the solution of \eqref{DWfg} when $f$ is non-zero, non-negative and compactly supported, and $f+g$ entirely vanishes. 
In particular, we study the large time behavior of the spatial null, critical, maximum and minimum sets of the solution in terms of the distance from the convex hull of the support of $f$. 
We call spatial minimum points of the solution of \eqref{DWfg} {\it time-delayed cold spots}.
We denote by $CS(f)$ and $d_f$ the convex hull and the diameter of the support of $f$ (see also Notation \ref{notation_main}).

For the spatial null set of the solution, we study its location by the following procedure:
\begin{description}
\item[Step I.1.] After a large time, the solution is negative on the parallel body of $CS(f)$ of radius $\sqrt{2nt} - d_f$ (Proposition \ref{prop_negativity_null}).
\item[Step I.2.] Let $\psi (t)$ be of small order of $t$. 
After a large time, the solution is positive on the region sandwiched between two parallel bodies of $CS(f)$ of radii $\sqrt{2nt}$ and $\psi (t)$ (Proposition \ref{prop_positivity_null}). 
\item[Step I.3.] After a large time, the solution is strictly increasing in each outer unit normal direction of $CS(f)$ on the region sandwiched between two parallel bodies of $CS(f)$ of radii $\sqrt{2nt}-d_f$ and $\sqrt{2nt}$ (Proposition \ref{prop_monotonicity_null}).  
\end{description}
Thus, after a large time, for each outer unit normal direction of $CS(f)$, the solution has spatial zeros in the region sandwiched between two parallel bodies of $CS(f)$ of radii $\sqrt{2nt}-d_f$ and $\sqrt{2nt}$. Implicit function theorem implies that the spatial null set contained in the sandwiched region is a smooth hyper surface homeomorphic to the $(n-1)$-dimensional sphere (Theorem \ref{null}).

For the spatial critical set of the solution, we study its location by the following procedure:
\begin{description}
\item[Step II.1.] After a large time, the solution is strictly increasing in each outer unit normal direction of $CS(f)$ on the region sandwiched between $CS(f)$ and the parallel body of $CS(f)$ of radius $\sqrt{(2n+4)t}-d_f$ (Proposition \ref{prop_positivity_crit}). 
\item[Step II.2.] Let $\psi (t)$ be of small order of $t$. 
After a large time, the solution is strictly decreasing in each outer unit normal direction of $CS(f)$ on the region sandwiched between two parallel bodies of $CS(f)$ of radii $\sqrt{(2n+4)t}$ and $\psi (t)$ (Proposition \ref{prop_negativity_crit}). 
\item[Step II.3.] After a large time, the solution is strictly concave in each outer unit normal direction of $CS(f)$ on the region sandwiched between two parallel bodies of $CS(f)$ of radii $\sqrt{(2n+4)t}-d_f$ and $\sqrt{(2n+4)t}$ (Proposition \ref{prop_concavity_crit}). 
\end{description}
Thus, after a large time, the solution has spatial critical points in $CS(f)$ and the region sandwiched between two parallel bodies of $CS(f)$ of radii $\sqrt{(2n+4)t}-d_f$ and $\sqrt{(2n+4)t}$. 
Implicit function theorem implies that the solution has at least two spatial critical points in the sandwiched region (Theorem \ref{crit1}).
Using the technique in [CK, Theorem 1], we also show that spatial critical points of the solution contained in $CS(f)$ converge to the centroid of $f$ as $t$ goes to infinity (Theorem \ref{crit2}).

For time-delayed hot spots, we show that all of them are contained in the region sandwiched between two parallel bodies of $CS(f)$ of radii $\sqrt{(2n+4)t}-d_f$ and $\sqrt{(2n+4)t}$ after a large time (Theorem \ref{min} (1)). Hence they escape from $CS(f)$.

For time-delayed cold spots, we show that all of them are contained in $CS(f)$ after a large time (Theorem \ref{min} (2)), that the set of time-delayed cold spots consists of one point after a large time (Corollary \ref{unique}), and that the unique time-delayed cold spot converges to the centroid of $f$ as $t$ goes to infinity (Corollary \ref{centroid}). 
Furthermore, we show that spatial maximum points of the absolute value of the solution are time-delayed cold spots after a large time.

Our argument is based on the so-called {\it Nishihara decomposition}.
It decomposes the solution of \eqref{DWfg} with $f=0$ into the diffusive part and the wave part (see Proposition \ref{solution}). 
It was introduced in [Nis] when $n=3$, in [MN] when $n=1$, and in [HO] when $n=2$.
For any dimensional case, Narazaki [Nar] gave a similar decomposition in terms of the Fourier transform.
In our previous results [SY], we gave the Nishihara decomposition for any space dimension without the Fourier transform.
Combining the Nishihara decomposition and the asymptotic expansion of the diffusive part, we can minutely analyze the shape of the solution.

The presentation of the proofs in sections 3, 4 and 5 is a bit lengthy. 
But the meaning of the word ``a large (enough) time'' will be clear. The large time depends on $n$, $d_f$, $\psi$, $\| f \|_1$, $\| f \|_{W^{\ast ,\infty}}$, the inradius (the radius of a maximal inscribed ball) of the support of $f$ and the mass of $f$ around an incenter (the center of a maximal inscribed ball) of the support of $f$.\\

\no
{\bf Acknowledgements.} 
The first author is partially supported by Waseda University Grant for Special Research Project 2014S-175.
The second author is partially supported by JSPS Kakenhi Grant Number 15J01600.

\begin{notation} 
{\rm 
Throughout this paper, we use the following notation:
\begin{itemize}
\item We denote the usual $L^p$ norm by $\|\cdot\|_p$, that is,
\[
\| \phi \|_p=
\begin{cases}
\ds \( \int_{\R^n} \lvert \phi (x) \rvert^p dx\)^{1/p} &( 1\leq p <\infty ),\\
\ds \esssup_{x \in \R^n} \lvert \phi (x) \rvert &(p=\infty) .
\end{cases}
\]
\item For a natural number $\ell$, we denote the Sobolev norm by $\| \cdot \|_{W^{\ell ,\infty}}$, that is,
\[
\| \phi \|_{W^{\ell ,\infty}} 
=\sum_{|\al |\leq \ell} \| \pd^\al \phi \|_\infty .
\]
\item Let $B^n$ and $S^{n-1}$ be the $n$-dimensional unit closed ball and the $(n-1)$-dimensional unit sphere, respectively.
\item For numbers $a$ and $b$, and for two sets $X$ and $Y$, we use the notation (Minkowski sum) $aX+bY = \left\{ ax+by \lvert x \in X,\ y\in Y \right\} \right.$. 
When $Y$ is a singleton $\{ y \}$, we write $aX+Y = aX+y$, for short. 
In particular, we write $B_t^n(x) =tB^n +x$ and $S_t^{n-1}(x) = tS^{n-1}+ x$, that is, the $n$-dimensional closed ball of radius $t$ centered at $x$ and the $(n-1)$-dimensional sphere of radius $t$ centered at $x$, respectively. 
\item For a set $X$ in $\R^n$, we denote by $X^c$, $X^\circ$ and $\bar{X}$ the complement, interior and closure of $X$, respectively.
\item Let $\sigma_n$ denote the $n$-dimensional Lebesgue surface measure.
\item The letter $r$ is always used for $r=\lvert x-y \rvert$.
\end{itemize}
}
\end{notation}

\begin{notation}
\label{notation_pre}
{\rm 
Let $I_\ell$ be the modified Bessel function of order $\ell$ defined in \eqref{Bessel}. We will use the following notation from Section 2:
\begin{enumerate}[(1)]
\item Let $n=1$. Put
\[
\tilde{k}_\ell (r,t) 
= t \( \frac{2}{\sqrt{t^2 -r^2}} \)^{\ell+1} I_{\ell+1} \( \frac{\sqrt{t^2 -r^2}}{2} \) -  2 \( \frac{2}{\sqrt{t^2 -r^2}} \)^\ell I_\ell \( \frac{\sqrt{t^2 -r^2}}{2} \) .
\]
\item Let $n$ be an odd number greater than one. Put
\begin{align*}
k_\ell (s)
&=\frac{1}{2^\ell}\sum_{j=0}^{\infty} \frac{1}{j!(j+\ell )!}\(\frac{s}{2}\)^{2j}
=\frac{I_\ell (s)}{s^\ell},\\
\tilde{k}_\ell (r,t) 
&= t k_{\ell +1} \( \frac{\sqrt{t^2 -r^2}}{2} \) -2 k_\ell \( \frac{\sqrt{t^2 -r^2}}{2} \) .
\end{align*}
\item Let $n$ be an even number. Put
\begin{align*}
k_\ell (s)
&=\sum_{j=0}^{\infty}\frac{1}{ \( 2(j+\ell ) \) !!(2j+1)!!}s^{2j+1}, \\
\tilde{k}_\ell (r,t) 
&= t k_{\ell +1} \( \frac{\sqrt{t^2 -r^2}}{2} \) -2 k_\ell \( \frac{\sqrt{t^2 -r^2}}{2} \) .
\end{align*}
\end{enumerate}
}
\end{notation}

\begin{notation}
\label{notation_main}
{\rm
Let us list up our notation for Sections 3, 4 and 5.
\begin{enumerate}
\item[$(f)$] Let $f$ be a non-zero non-negative smooth function with compact support. We denote by $CS(f)$, $d_f$ and $\rho_f$ the convex hull, the diameter and the inradius of the support of $f$, respectively. 
\item[$(A)$] For $0 \leq \rho_1 < \rho_2$, we denote by $A_f ( \rho_1, \rho_2 )$ the closed annulus 
\[
A_f \( \rho_1 , \rho_2 \) 
= \left. \left\{ x \in \overline{\R^n \sm CS(f)}  \rvert \rho_1 \leq \dist \( x ,   CS(f)  \) \leq \rho_2 \right\} .
\]
We denote by $A^\circ_f (\rho_1 , \rho_2 )$ the interior of $A_f ( \rho_1, \rho_2 )$.
\item[$(\ome)$] For every $\ome \in S^{n-1}$, let $\ome^\perp_\pm$ and $\Pos \{ \ome \}$ denote the half space $\{ x \in \R^n \vert x \cdot ( \pm \ome ) \geq 0 \}$ and the positive hull $\{ \rho \ome \vert \rho \geq 0 \}$, respectively.
\item[$( \rN )$] For each convex body (compact convex set with non-empty interior) $K$ in $\R^n$, let 
\[
\rN K = \left\{ (\xi ,\nu ) \lvert \xi \in \pd K ,\ \nu \in S^{n-1},\ \nu^\perp_- + \xi \supset K \right\} \right. .
\]
Every $(\xi ,\nu ) \in \rN K$ gives one of the unit outer normal directions of $K$. We will use the notion of $\rN K$ when $K = CS(f)$. Here, we remark that $CS(f) = \conv \overline{f^{-1} ( (0,+\infty ) )}$ is a convex body when $f$ satisfies the condition $(f)$. 
\end{enumerate}
}
\end{notation}

\section{Preliminaries}

\subsection{Expression of the solution}

Let $f$ be a smooth function, and $u$ the unique classical solution of \eqref{DWfg} with $f+g=0$. In this subsection, we list up the explicit form of $u$ and its derivatives. 

We denote by $S_n(t) g(x)$ the solution of \eqref{DWfg} with $f=0$. Then, the solution of \eqref{DWfg} is given by 
\begin{equation}
\label{sol_fg}
S_n (t) (f+g) (x) + \frac{\pd}{\pd t} S_n (t) f(x) .
\end{equation}
Thus, the solution of \eqref{DWfg} with $f+g=0$ is given by 
\begin{equation}
\label{sol_f-f}
u(x,t) = \frac{\pd}{\pd t} S_n (t) f(x) .
\end{equation}
The explicit form of $S_n (t) g(x)$ was given in [SY, Proposition 2.1].

\begin{rem}[{[SY, Remark 2.5]}]
\label{recursion}
{\rm 
\begin{enumerate}[$(1)$]
\item Let $n$ be an odd number greater than one. The kernel $k_\ell (s)$ has the following properties:
\[
k_{\ell +1} (s) = \frac{k_\ell'(s)}{s} ,\ k_\ell (0) = \frac{1}{2^\ell \ell !},\ k_\ell' (0) =0 .
\]
\item Let $n$ be an even number. The kernel $k_\ell (s)$ has the following properties:
\[
k_1(s)=\frac{\cosh(s)-1}{s},\ k_\ell (s)=\frac{k_{\ell-1}' (s)-k_{\ell-1}' (0)}{s} ,\ k_\ell (0)=0,\ k_\ell' (0)=\frac{1}{2^\ell \ell !}.
\]
\end{enumerate}
}
\end{rem}

The explicit form of $u$ directly follows from [SY, Proposition 2.8].

\begin{prop}
\label{solution}
Let $f$ be a smooth function, and $u$ the unique classical solution of \eqref{DWfg} with $f+g=0$.
\begin{enumerate}[$(1)$]
\item Let $n=1$. Put
\begin{align*}
\tilde{J}_1(t)f(x) = \frac{e^{-t/2}}{8} \int_{x-t}^{x+t} \tilde{k}_0 (r,t) f(y) dy ,\\
\widehat{\vect{W}}_1(t)f(x) = \frac{1}{2} \( f(x+t) +f(x-t) \) .
\end{align*}
Then, we have
\[
u(x,t)  = \tilde{J}_1(t)f(x)  +e^{-t/2} \widehat{\vect{W}}_1(t)f(x).
\]
\item Let $n$ be an odd number greater than one. Put
\begin{align*}
\ga_n 
&= \frac{c_n}{2^{n-1}}
=2^{-(3n-1)/2} \pi^{-(n-1)/2} ,\\
\tilde{J}_n(t) f(x) 
&=\frac{\ga_n e^{-t/2}}{4}\int_{B_t^n(x)} \tilde{k}_{\frac{n-1}{2}} (r,t) f(y)dy ,\\
\vect{W}_n(t)f(x)
&=c_n\sum_{j=0}^{(n-3)/2}\frac{1}{8^j j!}\(\frac{1}{t}\frac{\pd}{\pd t}\)^{(n-3)/2-j}\(\frac{1}{t}\int_{S_t^{n-1}(x)}f(y)d\sigma_{n-1}(y) \) ,\\
\widehat{\vect{W}}_n(t)f(x) 
&=\ga_n k_{\frac{n-1}{2}} (0) \int_{S_t^{n-1}(x)}f(y)d\sigma_{n-1}(y) ,\\
\widetilde{\vect{W}}_n(t)f(x)
&= \( \widehat{\vect{W}}_n(t) - \frac{1}{2} \vect{W}_n(t) + \frac{\pd}{\pd t} \vect{W}_n(t) \) f(x) .
\end{align*}
Then, we have
\[
u(x,t) 
=\tilde{J}_n(t)f(x) + e^{-t/2} \widetilde{\vect{W}}_n(t)f(x) .
\]
\item Let $n$ be an even number. Put
\begin{align*}
\ga_n
&=\frac{c_n}{2^{n-2}}
=2^{-(3n-2)/2} \pi^{-n/2} ,\\
\tilde{J}_n(t) f(x) 
&=\frac{\ga_n}{4} e^{-t/2} \int_{B_t^n(x)} \tilde{k}_{\frac{n}{2}} (r,t) f(y)dy, \\
\vect{W}_n(t)f(x)
&= 2 c_n \sum_{j=0}^{(n-2)/2}\frac{1}{8^j j!} \(\frac{1}{t}\frac{\pd}{\pd t}\)^{(n-2)/2-j} \int_{B_t^n(x)}\frac{1}{\sqrt{t^2-r^2}}f(y)dy ,\\
\widehat{\vect{W}}_n(t)f(x) 
&= \frac{\ga_n}{2} k'_{\frac{n}{2}} (0) t  \int_{B_t^n(x)}\frac{1}{\sqrt{t^2-r^2}} f(y) dy ,\\
\widetilde{\vect{W}}_n(t)f (x)
&=\( \widehat{\vect{W}}_n(t) - \frac{1}{2} \vect{W}_n(t) + \frac{\pd}{\pd t} \vect{W}_n(t) \) f(x) .
\end{align*}
Then, we have
\[
u(x,t)
=\tilde{J}_n(t)f(x) + e^{-t/2} \widetilde{\vect{W}}_n(t)f (x) .
\]
\end{enumerate}
\end{prop}

The explicit form of derivatives of $u$ directly follows from [SY, Lemma 3.5].

\begin{lem}
\label{derivative}
Let $f$ be a smooth function, and $u$ the unique classical solution of \eqref{DWfg} with $f+g=0$.
\begin{enumerate}[$(1)$]
\item Let $n=1$. Then, we have the following identities:
\begin{align*}
\frac{\pd u}{\pd x} (x,t) 
&= -\frac{e^{-t/2}}{32} \int_{x-t}^{x+t} \tilde{k}_1 (r,t) f(y) (x-y) dy \\
&\quad +\frac{e^{-t/2}}{16} (t-4)  \( f(x+t) -f(x-t) \) \\
&\quad + e^{-t/2}  \frac{\pd}{\pd x} \widehat{\vect{W}}_1(t)f(x) ,\\
\frac{\pd^2 u}{\pd x^2} (x,t) 
&= \frac{e^{-t/2}}{128} \int_{x-t}^{x+t} \(  \tilde{k}_2 (r,t) r^2 - 4 \tilde{k}_1 (r,t) \) f(y) dy \\
&\quad + \frac{e^{-t/2}}{256}t(t-8) \( f(x+t) +f(x-t) \) \\
&\quad +\frac{e^{-t/2}}{16} (t-4)  \( f' (x+t) -f' (x-t) \) \\
&\quad +e^{-t/2}  \frac{\pd^2}{\pd x^2} \widehat{\vect{W}}_1(t)f(x) .
\end{align*}
\item Let $n$ be an odd number greater than one. We have the following identities:
\begin{align*}
&\nabla u (x,t) \\
&= -\frac{\ga_n}{16} e^{-t/2} \int_{B^n_t(x)} \tilde{k}_{\frac{n+1}{2}} (r,t) f(y)  (x-y) dy \\
&\quad + \frac{\ga_n}{4} e^{-t/2} \( tk_{\frac{n+1}{2}} (0) - 2 k_{\frac{n-1}{2}} (0) \) t^{n-1} \int_{S^{n-1}} f(x+t\theta )  \theta  d \sigma_{n-1} (\theta ) \\
&\quad + e^{-t/2} \nabla \widetilde{\vect{W}}_n(t)f(x) ,\\
&\( \ome \cdot \nabla \)^2 u (x,t) \\
&= \frac{\ga_n}{64} e^{-t/2} \int_{B^n_t(x)} \( \tilde{k}_{\frac{n+3}{2}} (r,t) \( \ome \cdot (x-y) \)^2 -4 \tilde{k}_{\frac{n+1}{2}} (r,t)  \)  f(y) dy \\
&\quad +\frac{\ga_n}{4} e^{-t/2} \( tk_{\frac{n+1}{2}} (0) -  2 k_{\frac{n-1}{2}} (0) \) t^{n-1} \int_{S^{n-1}} \( \ome \cdot \nabla f(x+ t \theta ) \) \( \ome \cdot \theta \) d \sigma_{n-1} (\theta ) \\
&\quad + \frac{\ga_n}{16} e^{-t/2} \( tk_{\frac{n+3}{2}} (0) - 2 k_{\frac{n+1}{2}} (0) \) t^n \int_{S^{n-1}} f(x+t \theta ) \( \ome \cdot \theta \)^2 d \sigma_{n-1} (\theta ) \\
&\quad + e^{-t/2} \( \ome \cdot \nabla \)^2 \widetilde{\vect{W}}_n(t)f(x) ,
\end{align*}
where $\ome \in S^{n-1}$.
\item Let $n$ be an even number. We have the following identities:\footnote{[SY, Lemma 3.5] includes a typo: $k'_{\frac{n}{2}+3} (0)$ in the second term of $( \ome  \cdot \nabla)^2 \tilde{J}_n (t)f(x)$ should be replaced by $k'_{\frac{n}{2}+2} (0)$.} 
\begin{align*}
&\nabla  u (x,t) \\
&= -\frac{\ga_n}{16} e^{-t/2} \int_{B^n_t(x)} \tilde{k}_{\frac{n+2}{2}} (r,t) f(y)  (x-y) dy ,\\
&\quad +\frac{\ga_n}{8} e^{-t/2}   \( t k'_{\frac{n+2}{2}} (0) - 2k'_{\frac{n}{2}} (0) \) t^n \int_{B^n} \frac{1}{\sqrt{1 -\lvert z \rvert^2}} f(x+tz ) z dz \\
&\quad + e^{-t/2}   \nabla  \widetilde{\vect{W}}_n(t)f(x) ,\\
&\( \ome \cdot \nabla \)^2 u (x,t) \\
&= \frac{\ga_n}{64} e^{-t/2} \int_{B^n_t(x)} \(  \tilde{k}_{\frac{n+4}{2}} (r,t) \( \ome \cdot (x-y) \)^2 -  4 \tilde{k}_{\frac{n+2}{2}} (r,t) \)   f(y)  dy ,\\
&\quad + \frac{\ga_n}{8} e^{-t/2}  \( t k'_{\frac{n+2}{2}} (0) - 2k'_{\frac{n}{2}} (0) \)  t^n \int_{B^n} \frac{1}{\sqrt{1 -\lvert z \rvert^2}} \( \ome \cdot \nabla f(x+tz) \) \( \ome \cdot z \) dz \\
&\quad + \frac{\ga_n}{32} e^{-t/2}  \( t k'_{\frac{n+4}{2}} (0) - 2k'_{\frac{n+2}{2}} (0) \) t^{n+1}  \int_{B^n} \frac{1}{\sqrt{1 -\lvert z \rvert^2}} f(x+tz) \( \ome \cdot z \)^2 dz\\
&\quad + e^{-t/2} \( \ome \cdot \nabla \)^2 \widetilde{\vect{W}}_n(t)f(x) ,
\end{align*}
where $\ome \in S^{n-1}$.
\end{enumerate}
\end{lem}

\subsection{Expansion of the principal terms}

Let $f$ be a smooth function, and $u$ the unique classical solution of \eqref{DWfg} with $f+g=0$. We denote by $\mathcal{P} ( u(x,t) )$, $\mathcal{P} ( \ome \cdot \nabla u(x,t) )$ and $\mathcal{P} ( ( \ome \cdot \nabla )^2 u(x,t) )$ the first terms of $u(x,t)$, $\ome \cdot \nabla u(x,t)$ and  $( \ome \cdot \nabla )^2 u(x,t)$, respectively, in Proposition \ref{solution} and Lemma \ref{derivative}. For example, they are given by
\begin{align}
&\mathcal{P} \( u(x,t) \) 
= \tilde{J}_n (t) f(x) 
=\frac{\ga_n}{4} e^{-t/2} \int_{B_t^n(x)} \tilde{k}_{\frac{n}{2}} (r,t) f(y)dy ,\\
&\mathcal{P} \( \ome \cdot \nabla u(x,t) \) = 
\ds -\frac{\ga_n}{16} e^{-t/2} \int_{B^n_t(x)} \tilde{k}_{\frac{n+2}{2}} (r,t) f(y) \ome \cdot (x-y) dy ,\\
&\mathcal{P} \( \( \ome \cdot \nabla \)^2 u(x,t) \) 
=\ds \frac{\ga_n}{64} e^{-t/2} \int_{B^n_t(x)} \(  \tilde{k}_{\frac{n+4}{2}} (r,t) \( \ome \cdot (x-y) \)^2 -  4 \tilde{k}_{\frac{n+2}{2}} (r,t) \)   f(y)  dy 
\end{align}
when $n$ is an even number. We call $\mathcal{P} ( u(x,t) )$, $\mathcal{P} ( \ome \cdot \nabla u(x,t) )$ and $\mathcal{P} ( ( \ome \cdot \nabla )^2 u(x,t) )$ the {\it principal terms} of $u(x,t)$, $\ome \cdot \nabla u(x,t)$ and  $( \ome \cdot \nabla )^2 u(x,t)$, respectively, in what follows. They play important roles in the study of the large time behavior of $u(\cdot ,t)$. In this subsection, we give asymptotic expansions of the kernels of the principal terms.

\begin{rem}
\label{expansion_k}
{\rm 
\begin{enumerate}[$(1)$]
\item Let $n$ be an odd number greater than one. We have
\[
k_\ell (s) 
=\frac{1}{s^\ell} \frac{e^s}{\sqrt{2\pi s}} 
\(1-\frac{(\ell -1/2)(\ell +1/2)}{2s} +\frac{(\ell -3/2) (\ell -1/2) (\ell +1/2) (\ell +3/2)}{8s^2} +O\( \frac{1}{s^3} \) \) 
\]
as $s$ goes to infinity.
\item Let $n$ be an even number. We have
\[
k_\ell (s) =\frac{e^s}{2s^\ell}\(1-\frac{\ell (\ell -1)}{2s}+ \frac{(\ell -2 ) (\ell -1 ) \ell (\ell +1 )}{8s^2} + O\( \frac{1}{s^3} \)\) 
\]
as $s$ goes to infinity.
\end{enumerate}
}
\end{rem}

\begin{proof}
(1) This is a direct consequence of the fact \eqref{Bessel_3}.

(2) Using the recursion in Remark \ref{recursion}, we can show the expansion by induction.
\end{proof}

Remark \ref{expansion_k} directly implies the asymptotic expansions of the kernels of the principal terms with tedious computation.

\begin{lem}
\label{expansion_sqrt}
{\rm 
Let $\f (t)$ be of order of $\sqrt{t}$ as $t$ goes to infinity. Then, we have
\begin{enumerate}[$(1)$]
\item Let $n$ be an odd number. We have
\begin{align*}
\tilde{k}_\ell \( \f (t),t \) 
&= \frac{2^{\ell +1}}{\sqrt{\pi} t^{(2 \ell +1)/2}}  \exp \( \frac{\sqrt{t^2 -\f (t)^2}}{2} \) 
\( -\frac{2 \ell +1}{t} + \frac{1}{2} \( \frac{\f (t)}{t} \)^2 + \frac{\ell +2}{4} \( \frac{\f (t)}{t} \)^4  \right. \\
&\quad  \left. - \frac{3(2\ell +1)(2 \ell +3)\f (t)^2}{8t^3} + \frac{(2\ell -1)(2 \ell +1)(2 \ell +3)}{4t^2} + O \( \frac{1}{t^3} \)  \) 
\end{align*}
as $t$ goes to infinity.
\item Let $n$ be an even number. We have
\begin{align*}
\tilde{k}_\ell \( \f (t) ,t \)  
&= \frac{2^\ell}{t^\ell}  \exp \( \frac{\sqrt{t^2 -\f (t)^2}}{2} \) 
\( -\frac{2\ell}{t} + \frac{1}{2} \( \frac{\f (t)}{t} \)^2 + \frac{2\ell +3}{8} \( \frac{\f (t)}{t} \)^4 \right. \\
&\quad \left. - \frac{3\ell (\ell +1)\f (t)^2}{2t^3} + \frac{2\ell (\ell -1)(\ell +1)}{t^2} + O \( \frac{1}{t^3} \)  \) 
\end{align*}
as $t$ goes to infinity.
\end{enumerate}
}
\end{lem}

\begin{lem}
\label{expansion_smallo}
Let $\psi (t)$ be of small order of $t$ as $t$ goes to infinity. Then, we have
\begin{enumerate}[$(1)$]
\item Let $n$ be an odd number. We have
\begin{align*}
\tilde{k}_\ell \( \psi (t),t \) 
&= \frac{2^{\ell +1}}{\sqrt{\pi} t^{(2 \ell +1)/2}} \exp \( \frac{\sqrt{t^2 -\psi (t)^2}}{2} \)  \( -\frac{2 \ell +1}{t} + \frac{1}{2} \( \frac{\psi (t)}{t} \)^2 \right. \\
&\quad \left. + O \( \( \frac{\psi (t)}{t} \)^4 \) + O \( \frac{\psi (t)^2}{t^3} \) +  O \( \frac{1}{t^2} \)  \) 
\end{align*}
as $t$ goes to infinity.
\item Let $n$ be an even number. We have
\begin{align*}
\tilde{k}_\ell \( \psi (t) ,t \)  
&= \frac{2^\ell}{t^\ell}  \exp \( \frac{\sqrt{t^2 -\psi (t)^2}}{2} \)  \( -\frac{2\ell}{t} + \frac{1}{2} \( \frac{\psi (t)}{t} \)^2 \right. \\
&\quad \left. + O \( \( \frac{\psi (t)}{t} \)^4 \) + O \( \frac{\psi (t)^2}{t^3} \) + O \( \frac{1}{t^2} \)  \) 
\end{align*}
as $t$ goes to infinity.
\end{enumerate}
\end{lem}

\begin{lem}
\label{expansion_d}
Let $d>0$. 
\begin{enumerate}[$(1)$]
\item If $0\leq r \leq d$, then we have
\[
e^{-t/2} \tilde{k}_\ell (r,t) = 
\begin{cases}
\ds -\frac{(2\ell +1)2^{\ell +1}}{\sqrt{\pi} t^{(2\ell +3)/2}} \( 1+ O \( \frac{1}{t} \) \) &( n \in 2\N -1 ) ,\\
\ds -\frac{\ell  2^{\ell +1}}{t^{\ell +1}} \( 1+ O \( \frac{1}{t} \) \)  & (n \in 2 \N ) 
\end{cases}
\]
as $t$ goes to infinity.
\item For each $\ell \geq 0$, there exists a constant $T \geq d$ such that if $t \geq T$, then both of the following properties hold:
\begin{itemize}
\item For any $0 \leq r \leq d$, both $e^{-t/2} \tilde{k}_\ell (r,t)$ and $e^{-t/2} \tilde{k}_{\ell +1} (r,t)$ are negative. 
\item The function $[0,d] \ni r \mapsto e^{-t/2} \tilde{k}_\ell (r,t) \in (0,+\infty )$ is strictly increasing.
\end{itemize}
\end{enumerate}
\end{lem}

\begin{proof}
(1) This is a direct consequence of Lemma \ref{expansion_sqrt} and the fact \eqref{app_taylor2}. 

(2) The first property immediately follows from the first assertion. We give a proof of the second property for the even dimensional case. The argument works for the odd dimensional case. 

By the recursion in Remark \ref{recursion}, we have
\[
\frac{\pd \tilde{k}_\ell}{\pd r} (r,t)
=-\frac{r}{4} \tilde{k}_{\ell +1} (r,t) -\frac{r}{2\sqrt{t^2-r^2}} \( t k'_{\frac{\ell+1}{2}} (0) - 2 k'_{\frac{\ell}{2}} (0) \) 
\]
Thus, the first assertion implies
\[
e^{-t/2} \frac{\pd \tilde{k}_\ell}{\pd r} (r,t) 
=\frac{r (\ell +1) 2^{\ell +2}}{4t^{\ell +2}} \( 1+ O \( \frac{1}{t} \) \) 
\]
as $t$ goes to infinity. 
\end{proof}

\subsection{Estimates of the error terms}

Let $f$ be a smooth function, and $u$ the unique classical solution of \eqref{DWfg} with $f+g=0$. Put 
\begin{align*}
\mathcal{E} \( u(x,t) \) &= u(x,t) - \mathcal{P} \( u(x,t) \) ,\\
\mathcal{E} \( \ome \cdot \nabla u(x,t) \) &= \ome \cdot \nabla u(x,t) - \mathcal{P} \( \ome \cdot \nabla u(x,t) \) ,\\
\mathcal{E} \( \( \ome \cdot \nabla \)^2 u(x,t) \) &= \( \ome \cdot \nabla \)^2 u(x,t) - \mathcal{P} \( \( \ome \cdot \nabla \)^2 u(x,t) \) .
\end{align*}
We call  $\mathcal{E} ( u(x,t) )$, $\mathcal{E} ( \ome \cdot \nabla u(x,t) )$ and $\mathcal{E} ( ( \ome \cdot \nabla )^2 u(x,t) )$ the {\it error terms} of $u(x,t)$, $\ome \cdot \nabla u(x,t)$ and  $( \ome \cdot \nabla )^2 u(x,t)$, respectively, in what follows. We should pay attention to them in the study of the large time behavior of $u(\cdot ,t)$. In this subsection, we give estimates of the error terms.

\begin{lem}
\label{error}
Let $f$ be a smooth bounded function, and $u$ the unique classical solution of \eqref{DWfg} with $f+g=0$.
\begin{enumerate}[$(1)$]
\item Let $n$ be an odd number. There exists a positive constant $C(n)$ such that, for any $\ome \in S^{n-1}$, $x \in \R^n$ and $t>0$, the following inequalities hold:
\begin{align*}
\lvert \mathcal{E} \( u(x,t) \) \rvert 
&\leq C(n) e^{-t/2} (1+t)^{n-1} \| f \|_{W^{(n-1)/2,\infty}} ,\\
\lvert \mathcal{E} \( \ome \cdot \nabla u(x,t) \) \rvert
&\leq C(n) e^{-t/2} (1+t)^n \| f \|_{W^{(n+1)/2,\infty}} ,\\
\lvert \mathcal{E} \( \( \ome \cdot \nabla \)^2 u(x,t) \) \rvert
&\leq C(n) e^{-t/2} (1+t)^{n+1} \| f \|_{W^{(n+3)/2,\infty}} .
\end{align*}
\item Let $n$ be an even number. There exists a positive constant $C(n)$ such that, for any $\ome \in S^{n-1}$, $x \in \R^n$ and $t>0$, the following inequalities hold:
\begin{align*}
\lvert \mathcal{E} \( u(x,t) \) \rvert 
&\leq C(n) e^{-t/2} (1+t)^n \| f \|_{W^{n/2,\infty}} ,\\
\lvert \mathcal{E} \( \ome \cdot \nabla u(x,t) \) \rvert
&\leq C(n) e^{-t/2} (1+t)^{n+1} \| f \|_{W^{(n+2)/2,\infty}} ,\\
\lvert \mathcal{E} \( \( \ome \cdot \nabla \)^2 u(x,t) \) \rvert
&\leq C(n) e^{-t/2} (1+t)^{n+2} \| f \|_{W^{(n+4)/2,\infty}} .
\end{align*}
\end{enumerate}
\end{lem}

\begin{proof}
We give a proof for the even dimensional case. The argument works for the odd dimensional case.

From Proposition \ref{solution}, we have $\mathcal{E}  \( u(x,t) \) = e^{-t/2} \widetilde{\vect{W}}_n (t) f(x)$. Thus, [SY, Lemma 3.1] with $f+g=0$ and $\al =0$ implies the conclusion.

By Lemma \ref{derivative}, there exists a positive constant $C(n)$ such that, for any $\ome \in S^{n-1}$, $x \in \R^n$ and $t>0$, we have
\[
\lvert \mathcal{E} \( \ome \cdot \nabla u(x,t) \) \rvert
\leq C(n) e^{-t/2} (1+t)^{n+1} \| f \|_\infty + e^{-t/2}  \lvert \nabla \widetilde{\vect{W}}_n (t) f(x) \rvert .
\]
From [SY, Lemma 3.1] with $f+g=0$ and $\al =1$, we obtain the conclusion. 

The estimate of the error term of $( \ome \cdot \nabla)^2 u(x,t)$ is same as above.
\end{proof}

\section{Spatial zeros}

Let $f$ be as in Notation \ref{notation_main}, and $u$ the unique classical solution of \eqref{DWfg} with $f+g=0$. In this section, we discuss the large time behavior of the (non-trivial) spatial null set of $u$,
\begin{equation}
\mathcal{N} \( u(\cdot ,t ) \) = \left. \left\{ x \in \supp \( u \( \cdot ,t \) \)^\circ \rvert u(x,t) =0 \right\}  .
\end{equation}

\begin{prop}
\label{prop_negativity_null}
Let $f$ be as in Notation \ref{notation_main}.
There exists a constant $T \geq d_f^2 /(2n)$ such that if $t \geq T$, then, for any $x \in CS(f) + ( \sqrt{2nt} -d_f ) B^n$, we have $u(x,t) <0$.
\end{prop}

\begin{proof}
We give a proof for the even dimensional case. The argument works for the odd dimensional case.

We take $T \geq 2n$, which implies that if $t \geq T$, then, for any $x \in CS (f) + ( \sqrt{2nt} -d_f ) B^n$, the ball $B^n_t(x)$ contains $CS(f)$. 

By Proposition \ref{solution} and Lemma \ref{expansion_sqrt}, if $t$ is large enough, then we have 
\begin{align*}
\mathcal{P} \( u(x,t) \) 
&=\frac{\ga_n 2^{n/2}}{4 t^{n/2}} \int_{B^n_t(x)} \exp \( -\frac{r^2}{2\( t +\sqrt{t^2 -r^2} \)} \) \\
&\quad \times \( -\frac{n}{t} + \frac{1}{2} \( \frac{r}{t} \)^2 + \frac{n+3}{8} \( \frac{r}{t} \)^4  -\frac{3n(n+2)r^2}{8t^3}  + \frac{n(n-2)(n+2)}{4t^2}+ O \( \frac{1}{t^3} \) \) f(y) dy  \\
&<-\frac{n \ga_n 2^{n/2}}{4t^{n/2+2}} \int_{B^n_t (x)} \exp \( -\frac{r^2}{2\( t +\sqrt{t^2 -r^2} \)} \) \( 1+ O \( \frac{1}{t} \) \) f(y) dy \\
&<-\frac{n \ga_n 2^{n/2}}{8 t^{n/2+2}} \int_{B^n_t (x)} \exp \( -\frac{r^2}{2\( t +\sqrt{t^2 -r^2} \)} \)  f(y) dy \\
&<-\frac{n \ga_n 2^{n/2}}{8 t^{n/2+2}}  e^{-n} \| f \|_1 
\end{align*}
for any $x \in CS(f) + \sqrt{2nt} B^n$. 
Here, for the first inequality, we used the fact that if $t \geq n^2/4$, then, for any $0 \leq r \leq \sqrt{2nt}$, we have
\[
-\frac{n}{t} + \frac{1}{2} \( \frac{r}{t} \)^2 + \frac{n+3}{8} \( \frac{r}{t} \)^4  -\frac{3n(n+2)r^2}{8t^3}  + \frac{n(n-2)(n+2)}{4t^2}
\leq -\frac{n}{t^2} ,
\]
and equality holds if $r =\sqrt{2nt}$.
For the second inequality, we took a large enough $T$ such that if $t \geq T$, then we have $1+ O(1/t) >1/2$. The last inequality follows from $r \leq \sqrt{2nt}$. 

By Lemma \ref{error}, if $t$ is large enough, then we have
\[
\lvert \mathcal{E} \( u(x,t) \) \rvert
\leq C(n) e^{-t/2} (1+t)^n \| f \|_{W^{n/2, \infty}} 
< - \frac{n \ga_n 2^{n/2}}{16 t^{n/2+2}}  e^{-n} \| f \|_1
\]
for any $x \in CS(f) + \sqrt{2nt} B^n$.

Hence if $t$ is large enough, then we obtain
\[
u(x,t) <- \frac{n \ga_n 2^{n/2}}{16 t^{n/2+2}}  e^{-n} \| f \|_1  <0
\]
for any $x \in CS(f) + \sqrt{2nt} B^n$.
\end{proof}

\begin{lem}
\label{lem_positivity_null}
Let $f$ be as in Notation \ref{notation_main}.
There exists a constant $T >0$ such that if $t \geq T$, then, for any $x \in A_f ( \sqrt{2nt}, \sqrt{(2n+1)t})$, we have $u(x,t) >0$.
\end{lem}

\begin{proof}
We give a proof for the even dimensional case. The argument works for the odd dimensional case.

We take a large enough $T$ such that  if $t\geq T$, then we have $t\geq \sqrt{(2n+1)t} +d_f$, which implies that, for any $x \in CS(f) + \sqrt{(2n+1)t} B^n$, the ball $B^n_t(x)$ contains $CS(f)$. For every $x \in A_f ( \sqrt{2nt}, \sqrt{(2n+1)t})$, we take a point $\xi \in \pd  CS(f) $ such that $\dist (x, \pd CS(f) ) =\vert x- \xi \vert$. Let $\nu = (x-\xi )/ \vert x- \xi \vert$ and $\xi' = \xi -(\rho_f /2)\nu$. Let $i_f$ be an incenter of the support of $f$. We decompose the principal term of $u$ as 
\[
\mathcal{P} \( u(x,t) \) 
=\frac{\ga_n}{4}e^{-t/2} \( \int_{\supp f  \cap \( \nu^\perp_+ + \xi' \)} + \int_{\supp f \cap \( \nu^\perp_- + \xi' \)} \) \tilde{k}_{\frac{n}{2}} (r,t) f(y)dy 
=: P_1 + P_2 .
\]

By Lemma \ref{expansion_sqrt}, if $t$ is large enough, then we have
\begin{align*}
P_1
&=\frac{\ga_n 2^{n/2}}{4t^{n/2}} \int_{\supp f  \cap \( \nu^\perp_+ + \xi' \)} \exp \( -\frac{r^2}{2\( t+ \sqrt{t^2 -r^2} \)} \) \\
&\quad \times \( -\frac{n}{t} + \frac{1}{2}\( \frac{r}{t} \)^2  + \frac{n+3}{8} \( \frac{r}{t} \)^4 -\frac{3n(n+2)r^2}{8t^3}  + \frac{n(n-2)(n+2)}{4t^2}+ O \( \frac{1}{t^3} \) \) f(y) dy \\
&>- \frac{n \ga_n 2^{n/2}}{4t^{n/2+2}} \int_{\supp f  \cap \( \nu^\perp_+ + \xi' \)} \exp \( -\frac{r^2}{2\( t+ \sqrt{t^2 -r^2} \)} \)  \( 1+ O \( \frac{1}{t} \) \)  f(y) dy \\
&>- \frac{n \ga_n 2^{n/2}}{2t^{n/2+2}} \int_{\supp f  \cap \( \nu^\perp_+ + \xi' \)} \exp \( -\frac{r^2}{2\( t+ \sqrt{t^2 -r^2} \)} \)   f(y) dy\\
&>- \frac{n \ga_n 2^{n/2}}{2t^{n/2+2}} e^{-n/2} \| f \|_1
\end{align*}
for any $x \in A_f ( \sqrt{2nt}, \sqrt{(2n+1)t})$. 
Here, for the first inequality, we used the fact that, for any $r \geq \sqrt{2nt}$, we have
\[
-\frac{n}{t} + \frac{1}{2}\( \frac{r}{t} \)^2  + \frac{n+3}{8} \( \frac{r}{t} \)^4 -\frac{3n(n+2)r^2}{8t^3}  + \frac{n(n-2)(n+2)}{4t^2}
\geq -\frac{n}{t^2} ,
\]
and equality holds if $r= \sqrt{2nt}$. 
For the third inequality, we took a large enough $T$ such that if $t \geq T$, then we have $1+ O(1/t) <2$. 
The last inequality follows from $r \geq \sqrt{2nt}$.

By Lemma \ref{expansion_sqrt}, if $t$ is large enough, then we have
\begin{align*}
P_2
&=\frac{\ga_n 2^{n/2}}{4t^{n/2}} \int_{\supp f \cap \( \nu^\perp_- +\xi' \)} \exp \( -\frac{r^2}{2\( t+ \sqrt{t^2 -r^2} \)} \) \( -\frac{n}{t} + \frac{1}{2} \( \frac{r}{t} \)^2  + O \( \frac{1}{t^2} \) \) f(y) dy \\
&> \frac{\sqrt{n}\ga_n 2^{(n+1)/2}}{8t^{(n+3)/2}} \rho_f \int_{\supp f \cap \( \nu^\perp_- +\xi' \)} \exp \( -\frac{r^2}{2\( t+ \sqrt{t^2 -r^2} \)} \) \( 1 + O \( \frac{1}{\sqrt{t}} \) \) f(y) dy \\
&> \frac{\sqrt{n}\ga_n 2^{(n+1)/2}}{16t^{(n+3)/2}} \rho_f \int_{\supp f \cap \( \nu^\perp_- +\xi' \)} \exp \( -\frac{r^2}{2\( t+ \sqrt{t^2 -r^2} \)} \)  f(y) dy \\
&> \frac{\sqrt{n}\ga_n 2^{(n+1)/2}}{16t^{(n+3)/2}} e^{-(n+1)} \rho_f \int_{B^n_{\rho_f/2} \( i_f \)} f(y) dy 
\end{align*}
for any $x \in A_f ( \sqrt{2nt}, \sqrt{(2n+1)t})$. 
Here, the first inequality follows from $r \geq \sqrt{2nt}+\rho_f/2$. 
For the second inequality, we took a large enough $T$ such that if $t \geq T$, then we have $1+O(1/ \sqrt{t} ) >1/2$. 
For the last inequality, we took a large enough $T$ such that if $t \geq T$, then we have
\[
r \leq \sqrt{(2n+1)t} +d_f \leq \sqrt{(2n+2)t}.
\]
We also used Remark \ref{inscribed_ball} in the last inequality.

Therefore, if $t$ is large enough, then we obtain
\begin{align*}
\mathcal{P} \( u(x,t) \)  
&>- \frac{n \ga_n 2^{n/2}}{2t^{n/2+2}} e^{-n/2} \| f \|_1 
+\frac{\sqrt{n}\ga_n 2^{(n+1)/2}}{16t^{(n+3)/2}} e^{-(n+1)} \rho_f \int_{B^n_{\rho_f/2} \( i_f \)} f(y) dy \\
&>\frac{\sqrt{n}\ga_n 2^{(n+1)/2}}{32t^{(n+3)/2}} e^{-(n+1)} \rho_f \int_{B^n_{\rho_f/2} \( i_f \)} f(y) dy
\end{align*}
for any $x \in A_f ( \sqrt{2nt}, \sqrt{(2n+1)t})$.

By Lemma \ref{error}, if $t$ is large enough, then we have
\[
\lvert \mathcal{E} \( u(x,t) \) \rvert
\leq C(n) e^{-t/2} (1+t)^n \| f \|_{W^{n/2, \infty}} 
< \frac{\sqrt{n}\ga_n 2^{(n+1)/2}}{64t^{(n+3)/2}} e^{-(n+1)} \rho_f \int_{B^n_{\rho_f/2} \( i_f \)} f(y) dy
\]
for any $x \in A_f ( \sqrt{2nt}, \sqrt{(2n+1)t})$.

Hence if $t$ is large enough, then we obtain 
\[
u(x,t) >\frac{\sqrt{n}\ga_n 2^{(n+1)/2}}{64t^{(n+3)/2}} e^{-(n+1)} \rho_f \int_{B^n_{\rho_f/2} \( i_f \)} f(y) dy>0
\]
for any $x \in A_f ( \sqrt{2nt}, \sqrt{(2n+1)t})$.
\end{proof}

\begin{prop}
\label{prop_positivity_null}
Let $f$ be as in Notation \ref{notation_main}.
Let $\psi (t)$ be of small order of $t$ as $t$ goes to infinity with $\sqrt{(2n+1)t} \leq \psi (t) \leq t$ for any $t>0$.
There exists a constant $T >0$ such that if $t \geq T$, then, for any $x \in A_f (\sqrt{2nt} ,\psi (t) )$, we have $u(x,t) >0$.
\end{prop}

\begin{proof}
We give a proof for the even dimensional case. The argument works for the odd dimensional case.

We take a large enough $T$ such that  if $t\geq T$, then we have $t \geq \psi (t) +d_f$, which implies that, for any $x \in CS(f) + \psi (t) B^n$, the ball $B^n_t(x)$ contains $CS(f)$. 
By Lemma \ref{lem_positivity_null}, we show the existence of a constant $T >0$ such that if $t \geq T$, then, for any $x \in A_f (\sqrt{(2n+1)t} ,\psi (t) )$, we have $u(x,t) >0$. 

By Proposition \ref{solution} and Lemma \ref{expansion_smallo}, if $t$ is large enough, then we have
\begin{align*}
&\mathcal{P} \( u(x,t) \) \\
&=\frac{\ga_n 2^{n/2}}{4 t^{n/2}} \int_{B^n_t(x)} \exp \( -\frac{r^2}{2\( t +\sqrt{t^2 -r^2} \)} \)  \\
&\quad \times \( -\frac{n}{t} + \( \frac{2n}{2n+1} +\frac{1}{2n+1} \) \frac{1}{2} \( \frac{r}{t} \)^2  + O \( \( \frac{r}{t} \)^4 \) + O \( \frac{r^2}{t^3} \) + O \( \frac{1}{t^2} \) \) f(y) dy \\
&>\frac{\ga_n 2^{n/2}}{4 t^{n/2}} \int_{B^n_t(x)} \exp \( -\frac{r^2}{2\( t +\sqrt{t^2 -r^2} \)} \)  \\
&\quad \times  \( \frac{1}{2(2n+1)}  \( \frac{r}{t} \)^2   + O \( \( \frac{r}{t} \)^4 \) +O \( \frac{r^2}{t^3} \) + O \( \frac{1}{t^2} \) \) f(y) dy \\
&>\frac{\ga_n 2^{n/2}}{4 t^{n/2}} \int_{B^n_t(x)} \exp \( -\frac{r^2}{2\( t +\sqrt{t^2 -r^2} \)} \)   \frac{1}{4(2n+1)}  \( \frac{r}{t} \)^2   f(y) dy \\
&>\frac{\ga_n 2^{n/2}}{16 t^{n/2+1}} \int_{B^n_t(x)} \exp \( -\frac{r^2}{2\( t +\sqrt{t^2 -r^2} \)} \)    f(y) dy \\
&>\frac{\ga_n 2^{n/2}}{16 t^{n/2+1}} \exp \( -\frac{\( \psi (t)+d_f \)^2}{2t} \)  \| f \|_1 
\end{align*}
for any $x \in A_f (\sqrt{(2n+1)t} ,\psi (t) )$. 
Here, the first and third inequalities follow from $r \geq \sqrt{(2n+1)t}$. 
For the second inequality, we took a large enough $T$ such that if $t\geq T$, then we have
\[
\frac{1}{2(2n+1)}  \( \frac{r}{t} \)^2  + O \( \( \frac{r}{t} \)^4 \)+ O \( \frac{r^2}{t^3} \) + O \( \frac{1}{t^2} \) 
\geq \frac{1}{4(2n+1)}  \( \frac{r}{t} \)^2 
\]
for any $\sqrt{(2n+1)t} \leq r \leq \psi (t) +d_f$.
The last inequality follows from $r \leq \psi (t) +d_f$.

By Lemma \ref{error}, if $t$ is large enough, then we have
\[
\lvert \mathcal{E} \( u(x,t) \) \rvert
\leq C(n) e^{-t/2} (1+t)^n \| f \|_{W^{n/2, \infty}} 
< \frac{\ga_n 2^{n/2}}{32t^{n/2+1}} \exp \( -\frac{\( \psi (t)+d_f \)^2}{2t} \)   \| f \|_1 
\]
for any $x \in A_f (\sqrt{(2n+1)t} ,\psi (t) )$.

Hence if $t$ is large enough, then we obtain 
\[
u(x,t) >\frac{\ga_n 2^{n/2}}{32t^{n/2+1}} \exp \( -\frac{\( \psi (t)+d_f \)^2}{2t} \)   \| f \|_1 >0
\]
for any $x \in A_f (\sqrt{(2n+1)t} ,\psi (t) )$.
\end{proof}

\begin{prop}
\label{prop_monotonicity_null}
Let $f$ be as in Notation \ref{notation_main}.
There exists a constant $T \geq d_f^2 /(2n)$ such that if $t \geq T$, then, for any $(\xi ,\nu) \in \rN CS(f)$ and $x \in A_f (\sqrt{2nt}-d_f , \sqrt{2nt} ) \cap ( \Pos \{ \nu \} + \xi )$, we have $\nu \cdot \nabla u(x,t) >0$.
\end{prop}

\begin{proof}
We give a proof for the even dimensional case. The argument works for the odd dimensional case.

We take a large enough $T$ such that  if $t\geq T$, then we have $t \geq \sqrt{2nt} +d_f$, which implies that, for any $x \in CS(f) + \sqrt{2nt} B^n$, the ball $B^n_t(x)$ contains $CS(f)$.

By Lemmas \ref{derivative} and \ref{expansion_sqrt}, if $t$ is large enough, then we have
\begin{align*}
&\mathcal{P} \( \nu \cdot \nabla u (x,t) \) \\
&=-\frac{\ga_n 2^{n/2+1}}{16t^{n/2+1}} \int_{B^n_t (x)} \exp \( -\frac{r^2}{2 \( t+ \sqrt{t^2 -r^2} \)} \) 
\( -\frac{n+2}{t} + \frac{1}{2} \( \frac{r}{t} \)^2 + O \( \frac{1}{t^2} \)  \) f(y) \nu \cdot \( x-y \)  dy \\
&>\frac{3\ga_n 2^{n/2+1}}{32t^{n/2+2}} \int_{B^n_t (x)} \exp \( -\frac{r^2}{2 \( t+ \sqrt{t^2 -r^2} \)} \) 
\( 1  + O \( \frac{1}{t} \)  \)  f(y) \nu \cdot \( x-y \) dy \\
&>\frac{3\ga_n 2^{n/2+1}}{64t^{n/2+2}} \int_{B^n_t (x)} \exp \( -\frac{r^2}{2 \( t+ \sqrt{t^2 -r^2} \)} \) 
  f(y) \nu \cdot \( x-y \) dy \\
&>\frac{3\ga_n 2^{n/2+1}}{64t^{n/2+2}} \exp \( -\frac{2n+1}{2} \) \int_{B^n_t (x)}   f(y) \nu \cdot \( x-y \) dy \\
&>\frac{3\sqrt{2n-1} \ga_n  2^{n/2+1}}{64t^{n/2+3/2}} \exp \( -\frac{2n+1}{2} \) \| f \|_1 
\end{align*}
for any $(\xi ,\nu) \in \rN CS(f)$ and $x \in A_f (\sqrt{2nt}-d_f , \sqrt{2nt} ) \cap ( \Pos \{ \nu \} + \xi )$. 
Here, for the first and third inequalities, we took a large enough $T$ such that if $t\geq T$, then we have
\[
r \leq \sqrt{2nt} +d_f \leq \sqrt{(2n+1)t} .
\]
For the second inequality, we took a large enough $T$ such that if $t \geq T$, then we have $1+O(1/t) >1/2$.
For the last inequality, we took a large enough $T$ such that if $t\geq T$, then we have
\[
\nu \cdot (x-y) \geq \sqrt{2nt} -d_f \geq \sqrt{(2n-1)t} 
\]
for any $x \in A_f ( \sqrt{2nt} -d_f ,\sqrt{2nt} )$ and $y \in B^n_t(x)$.

By Lemma \ref{error}, if $t$ is large enough, then we have
\[
\lvert \mathcal{E} \( \nu \cdot \nabla u(x,t) \) \rvert
\leq C(n) e^{-t/2} (1+t)^{n+1} \| f \|_{W^{(n+2)/2, \infty}} 
< \frac{3\sqrt{2n-1} \ga_n 2^{n/2+1}}{128t^{n/2+3/2}} \exp \( -\frac{2n+1}{2} \) \| f \|_1 
\]
for any$(\xi ,\nu) \in \rN CS(f)$ and $x \in A_f (\sqrt{2nt}-d_f , \sqrt{2nt} ) \cap ( \Pos \{ \nu \} + \xi )$.

Hence if $t$ is large enough, then we obtain 
\[
\nu \cdot \nabla u(x,t) > \frac{3\sqrt{2n-1} \ga_n 2^{n/2+1}}{128t^{n/2+3/2}} \exp \( -\frac{2n+1}{2} \) \| f \|_1   >0
\]
for any $(\xi ,\nu) \in \rN CS(f)$ and $x \in A_f (\sqrt{2nt}-d_f , \sqrt{2nt} ) \cap ( \Pos \{ \nu \} + \xi )$.
\end{proof}

\begin{thm}
\label{null}
Let $f$ be as in Notation \ref{notation_main}.
Let $\psi (t)$ be of small order of $t$ as $t$ goes to infinity, and $\sqrt{(2n+1)t} \leq \psi (t) \leq t$ for any $t>0$. There exists a constant $T \geq d_f^2 /(2n)$ such that if $t \geq T$, we have
\[
\mathcal{N} \( u(\cdot ,t) \) \cap \( CS(f) + \psi (t) B^n \) 
=\bigcup_{( \xi ,\nu ) \in \rN CS (f)} \left. \left\{ x \in A^\circ_f \( \sqrt{2nt}-d_f , \sqrt{2nt} \) \cap \( \Pos \{ \nu \} + \xi \) \rvert u(x,t) =0 \right\} .
\]
and the following statements hold:
\begin{enumerate}[$(1)$]
\item If $n=1$, then the restriction of $u(\cdot ,t)$ to $CS(f) + \psi (t) B^1$ has just two zeros. 
\item If $n\geq 2$, then $\mathcal{N} ( u(\cdot ,t) ) \cap ( CS(f) + \psi (t) B^n )$ is a smooth hyper-surface homeomorphic to $S^{n-1}$.
\end{enumerate}
\end{thm}

\begin{proof}
From Propositions \ref{prop_negativity_null} and \ref{prop_positivity_null}, there exists a constant $T \geq d_f^2 /(2n)$ such that if $t \geq T$, then we have 
\[
\mathcal{N} \( u(\cdot ,t) \) \cap \( CS(f) + \psi (t) B^n \) \subset 
A^\circ_f \( \sqrt{2nt}-d_f , \sqrt{2nt} \) ,
\]

By Proposition \ref{prop_monotonicity_null}, if $t$ is large enough, then, for each $(\xi , \nu) \in \rN CS(f)$, we have a unique point $x \in A_f ( \sqrt{2nt}-d_f , \sqrt{2nt} ) \cap ( \Pos \{ \nu \} + \xi )$ with $u(x,t)=0$ and $\nu \cdot \nabla u(x,t) >0$.
By the correspondence in Remark \ref{correspondence}, implicit function theorem implies that the set $\mathcal{N} ( u(\cdot ,t) ) \cap ( CS(f) + \psi (t) B^n )$ is a smooth hyper-surface homeomorphic to $\pd  CS(f) $.
Since $\pd  CS(f)$ is homeomorphic to $S^{n-1}$, so is $\mathcal{N} ( u(\cdot ,t) ) \cap ( CS(f) + \psi (t) B^n )$.
\end{proof}

\section{Spatial critical points}

Let $f$ be as in Notation \ref{notation_main}, and $u$ the unique classical solution of \eqref{DWfg} with $f+g=0$. In this section, we discuss the large time behavior of the (non-trivial) spatial critical set of $u$,
\begin{equation}
\mathcal{C} \( u(\cdot ,t ) \) = \left. \left\{ x \in \supp \( u \( \cdot ,t \) \)^\circ \rvert \nabla u(x,t) =0 \right\} . 
\end{equation}

\begin{prop}
\label{prop_positivity_crit}
Let $f$ be as in Notation \ref{notation_main}.
There exists a constant $T \geq d_f^2 /(2n+4)$ such that if $t\geq T$, then, for any $(\xi ,\nu ) \in \rN CS(f)$ and $x \in A_f ( 0, \sqrt{(2n+4)t}-d_f ) \cap ( \Pos \{ \nu \} + \xi )$, we have $\nu \cdot \nabla u (x,t) >0$. 
\end{prop}

\begin{proof}
We give a proof for the even dimensional case. The argument works for the odd dimensional case.

We take $T\geq 2n+4$, which implies that if $t \geq T$, then, for any $x \in CS (f) + (\sqrt{(2n+4)t}-d_f)B^n$, the ball $B^n_t(x)$ contains $CS(f)$. Let $i_f$ be an incenter of the support of $f$.

By Lemmas \ref{derivative} and \ref{expansion_sqrt}, if $t$ is large enough, then we have
\begin{align*}
&\mathcal{P} \( \nu \cdot \nabla u(x,t) \) \\
&=-\frac{\ga_n 2^{n/2+1}}{16 t^{n/2+1}} \int_{B^n_t(x)} \exp \( -\frac{r^2}{2 \( t+\sqrt{t^2 -r^2} \)} \) \( -\frac{n+2}{t} + \frac{1}{2} \( \frac{r}{t} \)^2  +\frac{n+5}{8} \( \frac{r}{t} \)^4 \right. \\
&\quad \left. -\frac{3(n+2)(n+4)r^2}{8t^3} + \frac{n(n+2)(n+4)}{4t^2} +  O\( \frac{1}{t^3} \) \)   f(y) \nu \cdot \( x-y \) dy \\
&> \frac{(n+2) \ga_n 2^{n/2+1}}{16 t^{n/2+3}} \int_{B^n_t(x)} \exp \( -\frac{r^2}{2 \( t+\sqrt{t^2 -r^2} \)} \) \( 1+ O \( \frac{1}{t} \) \) f(y) \nu \cdot (x-y ) dy \\
&> \frac{(n+2) \ga_n 2^{n/2+1}}{32 t^{n/2+3}} \int_{B^n_t(x)} \exp \( -\frac{r^2}{2 \( t+\sqrt{t^2 -r^2} \)} \)  f(y) \nu \cdot (x-y ) dy \\
&> \frac{(n+2) \ga_n 2^{n/2+1}}{32 t^{n/2+3}} e^{-(n+2)} \int_{B^n_t(x)} f(y) \nu \cdot (x-y ) dy \\
&> \frac{(n+2) \ga_n 2^{n/2+1}}{64 t^{n/2+3}} e^{-(n+2)} \rho_f \int_{B^n_{\rho_f /2} \( i_f \)}  f(y)  dy
\end{align*}
for any $(\xi ,\nu ) \in \rN CS(f)$ and $x \in A_f ( 0, \sqrt{(2n+4)t}-d_f ) \cap ( \Pos \{ \nu \} + \xi )$.
Here, for the first inequality, we used the fact that if $t \geq (n+2)^2/4$, then, for any $ 0 \leq r \leq \sqrt{(2n+4)t}$, we have
\[
-\frac{n+2}{t} + \frac{1}{2} \( \frac{r}{t} \)^2  +\frac{n+5}{8} \( \frac{r}{t} \)^4 -\frac{3(n+2)(n+4)r^2}{8t^3} + \frac{n(n+2)(n+4)}{4t^2}
\leq - \frac{n+2}{t^2} ,
\]
and equality holds if $r=\sqrt{(2n+4)t}$. 
The second inequality is usual. The third inequality follows from $r \leq \sqrt{(2n+4)t}$. 
The last inequality follows from Remark \ref{inscribed_ball} and $\nu \cdot (x-y) \geq \rho_f /2$. 

By Lemma \ref{error}, if $t$ is large enough, then we have
\[
\lvert \mathcal{E} \( \nu \cdot \nabla u(x,t) \) \rvert 
\leq C(n) e^{-t/2} (1+t)^{n+1} \| f \|_{W^{(n+2)/2,\infty}} 
<\frac{(n+2) \ga_n 2^{n/2+1}}{128 t^{n/2+3}} e^{-(n+2)} \rho_f \int_{B^n_{\rho_f /2} \( i_f \)}  f(y)  dy
\]
for any $(\xi ,\nu ) \in \rN CS(f)$ and $x \in A_f ( 0, \sqrt{(2n+4)t}-d_f ) \cap ( \Pos \{ \nu \} + \xi )$.

Hence if $t$ is large enough, then we obtain
\[
\nu \cdot \nabla u(x,t) 
> \frac{(n+2) \ga_n 2^{n/2+1}}{128 t^{n/2+3}} e^{-(n+2)} \rho_f \int_{B^n_{\rho_f /2} \( i_f \)}  f(y)  dy
>0 
\]
for any $(\xi ,\nu ) \in \rN CS(f)$ and $x \in A_f ( 0, \sqrt{(2n+4)t}-d_f ) \cap ( \Pos \{ \nu \} + \xi )$.
\end{proof}

\begin{lem}
\label{lem_negativity_crit}
Let $f$ be as in Notation \ref{notation_main}.
There exists a constant $T >0$ such that if $t\geq T$, then, for any $(\xi ,\nu ) \in \rN CS(f)$ and $x \in A_f ( \sqrt{(2n+4)t} , \sqrt{(2n+5)t} )  \cap ( \Pos \{ \nu \} + \xi )$, we have $\nu \cdot \nabla u (x,t) <0$. 
\end{lem}

\begin{proof}
We give a proof for the even dimensional case. The argument works for the odd dimensional case.

We take a large enough $T$ such that if $t \geq T$, then we have $t \geq \sqrt{(2n+5)t} +d_f$, which implies that, for any $x \in CS (f) + \sqrt{(2n+5)t}B^n$, the ball $B^n_t(x)$ contains $CS(f)$. 
For every $(\xi ,\nu ) \in \rN CS(f)$, put $\xi' = \xi -(\rho_f /2)\nu$. 
Let $i_f$ be an incenter of the support of $f$. We decompose the principal term of $\nu \cdot \nabla u$ as 
\begin{align*}
\mathcal{P} \( \nu \cdot u(x,t) \) 
&=-\frac{\ga_n}{16}e^{-t/2} \( \int_{\supp f \cap \( \nu^\perp_+ + \xi' \)} + \int_{\supp f \cap \( \nu^\perp_- + \xi' \)} \) \tilde{k}_{\frac{n+2}{2}} (r,t) f(y) \nu \cdot (x-y) dy \\
&=: P_1' + P_2' .
\end{align*}

By Lemma \ref{expansion_sqrt}, if $t$ is large enough, then we have
\begin{align*}
P_1'
&=-\frac{\ga_n 2^{n/2+1}}{16 t^{n/2+1}} \int_{\supp f  \cap \( \nu^\perp_+ + \xi' \)} \exp \( -\frac{r^2}{2 \( t+\sqrt{t^2 -r^2} \)} \) \( -\frac{n+2}{t} + \frac{1}{2} \( \frac{r}{t} \)^2   \right. \\ 
&\quad \left. +\frac{n+5}{8} \( \frac{r}{t} \)^4 -\frac{3(n+2)(n+4)r^2}{8t^3} + \frac{n(n+2)(n+4)}{4t^2} +  O\( \frac{1}{t^3} \) \)   f(y) \nu \cdot \( x-y \) dy \\
&< \frac{(n+2) \ga_n 2^{n/2+1}}{16 t^{n/2+3}} \int_{\supp f  \cap \( \nu^\perp_+ + \xi' \)} \exp \( -\frac{r^2}{2 \( t+\sqrt{t^2 -r^2} \)} \)  \( 1 +  O\( \frac{1}{t} \) \)   f(y) \nu \cdot \( x-y \) dy\\
&< \frac{(n+2) \ga_n 2^{n/2+1}}{8 t^{n/2+3}} \int_{\supp f  \cap \( \nu^\perp_+ + \xi' \)} \exp \( -\frac{r^2}{2 \( t+\sqrt{t^2 -r^2} \)} \)    f(y) \nu \cdot \( x-y \) dy \\
&< \frac{(n+2) \ga_n 2^{n/2+1}}{8 t^{n/2+3}} e^{-(n+2)/2} \int_{\supp f  \cap \( \nu^\perp_+ + \xi' \)}   f(y) \nu \cdot \( x-y \) dy \\
&< \frac{(n+2) \sqrt{2n+6} \ga_n 2^{n/2+1}}{8 t^{n/2+5/2}} e^{-(n+2)/2} \| f \|_1
\end{align*}
for any $(\xi ,\nu ) \in \rN CS(f)$ and $x \in A_f ( \sqrt{(2n+4)t} , \sqrt{(2n+5)t} )  \cap ( \Pos \{ \nu \} + \xi )$. 
Here, for the first inequality, we used the fact that, for any $r \geq \sqrt{(2n+4)t}$, we have
\[
-\frac{n+2}{t} + \frac{1}{2} \( \frac{r}{t} \)^2  +\frac{n+5}{8} \( \frac{r}{t} \)^4 -\frac{3(n+2)(n+4)r^2}{8t^3} + \frac{n(n+2)(n+4)}{4t^2}
\geq -\frac{n+2}{t^2} ,
\]
and equality holds if $r= \sqrt{(2n+4)t}$. 
The second inequality is usual. 
The third inequality follows from $r \geq \sqrt{(2n+4)t}$.
For the last inequality, we took a large enough $T$ such that if $t \geq T$, then we have
\[
\nu \cdot (x-y) 
\leq r 
\leq \sqrt{(2n+5)t} +d_f 
\leq \sqrt{(2n+6)t} .
\]

By Lemma \ref{expansion_sqrt}, if $t$ is large enough, then we have
\begin{align*}
P_2'
&=- \frac{\ga_n 2^{n/2+1}}{16t^{n/2+1}} \int_{\supp f \cap \( \nu^\perp_- +\xi' \)} \exp \( -\frac{r^2}{2\( t+ \sqrt{t^2 -r^2} \)} \) \\
&\quad \times  \( -\frac{n+2}{t} + \frac{1}{2} \( \frac{r}{t} \)^2  + O \( \frac{1}{t^2} \) \) f(y) \nu \cdot (x-y) dy \\
&< -\frac{\sqrt{2n+4}\ga_n 2^{n/2+1}}{32 t^{n/2+5/2}} \rho_f \int_{\supp f \cap \( \nu^\perp_- +\xi' \)} \exp \( -\frac{r^2}{2\( t+ \sqrt{t^2 -r^2} \)} \)  \( 1 + O \( \frac{1}{\sqrt{t}} \) \) f(y) \nu \cdot (x-y)  dy \\
&< -\frac{\sqrt{2n+4}\ga_n 2^{n/2+1}}{64 t^{n/2+5/2}} \rho_f \int_{\supp f \cap \( \nu^\perp_- +\xi' \)} \exp \( -\frac{r^2}{2\( t+ \sqrt{t^2 -r^2} \)} \)  f(y) \nu \cdot (x-y)  dy \\
&< -\frac{\sqrt{2n+4}\ga_n 2^{n/2+1}}{64 t^{n/2+5/2}}  e^{-(n+3)} \rho_f \int_{\supp f \cap \( \nu^\perp_- +\xi' \)} f(y) \nu \cdot (x-y)  dy \\
&< -\frac{(n+2) \ga_n 2^{n/2+1}}{32 t^{n/2+2}} e^{-(n+3)} \rho_f \int_{B^n_{\rho_f /2} \( i_f \)}  f(y)  dy 
\end{align*}
for any  $(\xi ,\nu ) \in \rN CS(f)$ and $x \in A_f ( \sqrt{(2n+4)t} , \sqrt{(2n+5)t} )  \cap ( \Pos \{ \nu \} + \xi )$. 
Here, the first inequality follows from $r \geq \sqrt{(2n+4)t}+\rho_f/2$. 
The second inequality is usual. 
The third inequality follows from
\[ 
r \leq \sqrt{(2n+5)t} +d_f \leq \sqrt{(2n+6)t} ,
\]
which has been assumed in the last one for $P_1'$. 
The last inequality follows from $\nu \cdot (x-y) \geq \sqrt{(2n+4)t}$. 
We also used Remark \ref{inscribed_ball} in the last inequality.

Therefore, if $t$ is large enough, then we obtain
\begin{align*}
\mathcal{P} \( \nu \cdot \nabla  u(x,t) \)  
&< \frac{(n+2) \sqrt{2n+6} \ga_n 2^{n/2+1}}{8 t^{n/2+5/2}} e^{-(n+2)/2} \| f \|_1  \\
&\quad -\frac{(n+2) \ga_n 2^{n/2+1}}{32 t^{n/2+2}} e^{-(n+3)} \rho_f \int_{B^n_{\rho_f /2} \( i_f \)}  f(y)  dy   \\
&<-\frac{(n+2) \ga_n 2^{n/2+1}}{64 t^{n/2+2}} e^{-(n+3)} \rho_f \int_{B^n_{\rho_f /2} \( i_f \)}  f(y)  dy
\end{align*}
for any  $(\xi ,\nu ) \in \rN CS(f)$ and $x \in A_f ( \sqrt{(2n+4)t} , \sqrt{(2n+5)t} )  \cap ( \Pos \{ \nu \} + \xi )$.

By Lemma \ref{error}, if $t$ is large enough, then we have
\[
\lvert \mathcal{E} \( \nu \cdot \nabla u(x,t) \) \rvert
\leq C(n) e^{-t/2} (1+t)^{n+1} \| f \|_{W^{(n+2)/2, \infty}} 
< \frac{(n+2) \ga_n 2^{n/2+1}}{128 t^{n/2+2}} e^{-(n+3)}  \rho_f \int_{B^n_{\rho_f /2} \( i_f \)}  f(y)  dy
\]
for any  $(\xi ,\nu ) \in \rN CS(f)$ and $x \in A_f ( \sqrt{(2n+4)t} , \sqrt{(2n+5)t} )  \cap ( \Pos \{ \nu \} + \xi )$.

Hence if $t$ is large enough, then we obtain 
\[
\nu \cdot \nabla u(x,t) 
<-\frac{(n+2) \ga_n 2^{n/2+1}}{128 t^{n/2+2}} e^{-(n+3)} \rho_f \int_{B^n_{\rho_f /2} \( i_f \)}  f(y)  dy 
\]
for any  $(\xi ,\nu ) \in \rN CS(f)$ and $x \in A_f ( \sqrt{(2n+4)t} , \sqrt{(2n+5)t} )  \cap ( \Pos \{ \nu \} + \xi )$.
\end{proof}

\begin{prop}
\label{prop_negativity_crit}
Let $f$ be as in Notation \ref{notation_main}.
Let $\psi (t)$ be of small order of  $t$ as $t$ goes to infinity with $\sqrt{(2n+5)t} \leq \psi (t) \leq t$ for any $t>0$. There exists a constant $T >0$ such that if $t\geq T$, then, for any $(\xi ,\nu ) \in \rN CS(f)$ and $x \in A_f ( \sqrt{(2n+4)t}, \psi (t) ) \cap ( \Pos \{ \nu \} + \xi )$, we have $\nu \cdot \nabla u (x,t) <0$. 
\end{prop}

\begin{proof}
We give a proof for the even dimensional case. The argument works for the odd dimensional case.

We take a large enough $T$ such that  if $t\geq T$, then we have $t \geq \psi (t) +d_f$, which implies that, for any $x \in CS(f) + \psi (t) B^n$, the ball $B^n_t(x)$ contains $CS(f)$. 
By Lemma \ref{lem_negativity_crit}, we show the existence of a constant $T >0$ such that if $t \geq T$, then, for any $(\xi ,\nu ) \in \rN CS(f)$ and $x \in A_f ( \sqrt{(2n+5)t}, \psi (t) ) \cap ( \Pos \{ \nu \} +\xi )$, we have $\nu \cdot \nabla u(x,t) <0$. 

By Lemmas \ref{derivative} and \ref{expansion_smallo}, if $t$ is large enough, then we have
\begin{align*}
&\mathcal{P} \( \nu \cdot \nabla u(x,t) \) \\
&=-\frac{\ga_n 2^{n/2+1}}{16 t^{n/2+1}} \int_{B^n_t(x)} \exp \( -\frac{r^2}{2\( t +\sqrt{t^2 -r^2} \)} \) \\ 
&\quad \times \( -\frac{n+2}{t} + \( \frac{2n+4}{2n+5} +\frac{1}{2n+5} \) \frac{1}{2} \( \frac{r}{t} \)^2 
+ O \( \( \frac{r}{t} \)^4 \) + O\( \frac{r^2}{t^3} \) + O \( \frac{1}{t^2} \) \) f(y) \nu \cdot (x-y) dy \\
&< -\frac{\ga_n 2^{n/2+1}}{16 t^{n/2+1}} \int_{B^n_t(x)} \exp \( -\frac{r^2}{2\( t +\sqrt{t^2 -r^2} \)} \)  \\
&\quad \times \( \frac{1}{2(2n+5)}  \( \frac{r}{t} \)^2   + O \( \( \frac{r}{t} \)^4 \) + O\( \frac{r^2}{t^3} \) + O \( \frac{1}{t^2} \) \) f(y) \nu \cdot (x-y) dy \\
&<-\frac{\ga_n 2^{n/2+1}}{16 t^{n/2+1}} \int_{B^n_t(x)} \exp \( -\frac{r^2}{2\( t +\sqrt{t^2 -r^2} \)} \)   \frac{1}{4(2n+5)}  \( \frac{r}{t} \)^2   f(y) \nu \cdot (x-y) dy \\
&< -\frac{\ga_n 2^{n/2+1}}{64 t^{n/2+2}} \int_{B^n_t(x)} \exp \( -\frac{r^2}{2\( t +\sqrt{t^2 -r^2} \)} \)    f(y) \nu \cdot (x-y) dy \\
&< - \frac{\sqrt{2n+5} \ga_n 2^{n/2+1}}{64 t^{n/2+3/2}} \exp \( -\frac{\( \psi (t)+d_f \)^2}{2t} \)  \| f \|_1 
\end{align*}
for any $(\xi ,\nu ) \in \rN CS(f)$ and $x \in A_f ( \sqrt{(2n+5)t}, \psi (t) ) \cap ( \Pos \{ \nu \} +\xi )$.
Here, the first and third inequalities follow from $r \geq \sqrt{(2n+5)t}$. 
For the second inequality, we took a large enough $T$ such that if $t\geq T$, then, for any $\sqrt{(2n+5)t} \leq r \leq \psi (t) +d_f$, we have
\[
\frac{1}{2(2n+5)}  \( \frac{r}{t} \)^2  + O \( \( \frac{r}{t} \)^4 \) +O\( \frac{r^2}{t^3} \) + O \( \frac{1}{t^2} \) 
\geq \frac{1}{4(2n+5)}  \( \frac{r}{t} \)^2 .
\]
The last inequality follows from 
\[
\sqrt{(2n+5)t} \leq \nu \cdot (x-y) \leq r \leq \psi (t) +d_f.
\]

By Lemma \ref{error}, if $t$ is large enough, then we have
\[
\lvert \mathcal{E} \( \nu \cdot u(x,t) \) \rvert
\leq C(n) e^{-t/2} (1+t)^{n+1} \| f \|_{W^{(n+2)/2, \infty}} 
< \frac{\sqrt{2n+5} \ga_n 2^{n/2+1}}{128 t^{n/2+3/2}} \exp \( -\frac{\( \psi (t)+d_f \)^2}{2t} \)  \| f \|_1
\]
for any $(\xi ,\nu ) \in \rN CS(f)$ and $x \in A_f ( \sqrt{(2n+5)t}, \psi (t) ) \cap ( \Pos \{ \nu \} +\xi )$.

Hence if $t$ is large enough, then we obtain 
\[
\nu \cdot \nabla u(x,t) 
<- \frac{\sqrt{2n+5} \ga_n 2^{n/2+1}}{128 t^{n/2+3/2}} \exp \( -\frac{\( \psi (t)+d_f \)^2}{2t} \)  \| f \|_1 
<0
\]
for any $(\xi ,\nu ) \in \rN CS(f)$ and $x \in A_f ( \sqrt{(2n+5)t}, \psi (t) ) \cap ( \Pos \{ \nu \} +\xi )$.
\end{proof}

\begin{prop}
\label{prop_concavity_crit}
Let $f$ be as in Notation \ref{notation_main}.
There exists a constant $T \geq d_f^2 / (2n+4)$ such that if $t\geq T$, then, for any $(\xi ,\nu ) \in \rN CS(f)$ and $x \in A_f (\sqrt{(2n+4)t} -d_f ,\sqrt{(2n+4)t} ) \cap ( \Pos \{ \nu \} + \xi )$, we have $( \nu \cdot \nabla )^2 u(x,t)  <0$. 
\end{prop}

\begin{proof}
We give a proof for the even dimensional case. The argument works for the odd dimensional case.

We take a large enough $T$ such that if $t\geq T$, then we have $t \geq \sqrt{(2n+4)t} +d_f$, which implies that, for any $x \in CS(f) + \sqrt{(2n+4)t} B^n$, the ball $B^n_t(x)$ contains $CS(f)$.

By Lemmas \ref{derivative} and \ref{expansion_sqrt}, if $t$ is large enough, then we have
\begin{align*}
\mathcal{P} \( \( \nu \cdot \nabla \)^2 u (x,t) \) 
&=\frac{\ga_n 2^{n/2+3}}{64t^{n/2+1}} \int_{B^n_t (x)} \exp \( -\frac{r^2}{2 \( t+ \sqrt{t^2 -r^2} \)} \) \\
&\quad \times \( \frac{n+2}{t} -\frac{r^2 + (n+4) \( \nu \cdot (x-y) \)^2}{2t^2} + \frac{r^2 \( \nu \cdot (x-y) \)^2}{4t^3} + O \( \frac{1}{t^2} \)  \)  f(y) dy \\
&<-\frac{(3n+2) \ga_n 2^{n/2+3}}{256t^{n/2+2}} \int_{B^n_t (x)} \exp \( -\frac{r^2}{2 \( t+ \sqrt{t^2 -r^2} \)} \) \( 1+ O \( \frac{1}{t} \) \)  f(y) dy\\
&<-\frac{(3n+2) \ga_n 2^{n/2+3}}{512t^{n/2+2}} \int_{B^n_t (x)} \exp \( -\frac{r^2}{2 \( t+ \sqrt{t^2 -r^2} \)}  \)  f(y) dy\\
&<-\frac{(3n+2) \ga_n 2^{n/2+3}}{512t^{n/2+2}} \exp \( -\frac{2n+5}{2} \) \| f \|_1 
\end{align*}
for any $(\xi ,\nu ) \in \rN CS(f)$ and $x \in A_f (\sqrt{(2n+4)t} -d_f ,\sqrt{(2n+4)t} ) \cap ( \Pos \{ \nu \} + \xi )$. 
Here, for the first and last inequalities, we took a large enough $T$ such that if $t\geq T$, then we have
\[
\sqrt{\( 2n+\frac{7}{2} \)t}
\leq \sqrt{(2n+4)t} -d_f
\leq \nu \cdot (x-y)
\leq r
\leq \sqrt{(2n+4)t} +d_f 
\leq \sqrt{(2n+5)t} .
\]
The second inequality is usual.

By Lemma \ref{error}, if $t$ is large enough, then we have
\[
\lvert \mathcal{E} \( \( \nu \cdot \nabla \)^2 u(x,t) \) \rvert
\leq C(n) e^{-t/2} (1+t)^{n+2} \| f \|_{W^{(n+4)/2, \infty}} 
< \frac{(3n+2) \ga_n 2^{n/2+3}}{1024t^{n/2+2}} \exp \( -\frac{2n+5}{2} \) \| f \|_1 
\]
for any $(\xi ,\nu ) \in \rN CS(f)$ and $x \in A_f (\sqrt{(2n+4)t} -d_f ,\sqrt{(2n+4)t} ) \cap ( \Pos \{ \nu \} + \xi )$.

Hence if $t$ is large enough, then we obtain 
\[
\( \nu \cdot \nabla \)^2 u(x,t) 
< - \frac{(3n+2) \ga_n 2^{n/2+3}}{1024t^{n/2+2}} \exp \( -\frac{2n+5}{2} \) \| f \|_1 
<0
\]
for any $(\xi ,\nu ) \in \rN CS(f)$ and $x \in A_f (\sqrt{(2n+4)t} -d_f ,\sqrt{(2n+4)t} ) \cap ( \Pos \{ \nu \} + \xi )$.
\end{proof}

\begin{thm}
\label{crit1}
Let $f$ be as in Notation \ref{notation_main}.
Let $\psi (t)$ be of small order of $t$ as $t$ goes to infinity with $\sqrt{(2n+5)t} \leq \psi (t) \leq t$ for any $t>0$. 
There exists a constant $T \geq d_f^2 /(2n+4)$ such that if $t \geq T$, then we have
\[
\mathcal{C} \( u(\cdot ,t ) \) \cap \( CS(f) + \psi (t) B^n \) 
\subset A^\circ_f \( \sqrt{(2n+4)t} -d_f , \sqrt{(2n+4)t}  \) \cup CS(f) ,
\]
and the following statements hold:
\begin{enumerate}[$(1)$]
\item $u (\cdot ,t)$ has at least one critical point in $CS(f)$.  
\item If $n=1$, then the restriction of $u(\cdot ,t)$ to $A_f ( \sqrt{6t} -d_f , \sqrt{6t}  )$ has just two critical points. If $n \geq 2$, then the restriction of $u(\cdot ,t)$ to $A_f ( \sqrt{(2n+4)t} -d_f , \sqrt{(2n+4)t}  )$ has at least two critical points.
\end{enumerate}
\end{thm}

\begin{proof}
From Propositions \ref{prop_positivity_crit} and \ref{prop_negativity_crit}, there exists a constant $T \geq d_f^2 /(2n+4)$ such that if $t \geq T$, then we have
\[
\mathcal{C} \( u(\cdot ,t ) \) \cap \( CS(f) + \psi (t) B^n \) 
\subset A^\circ_f \( \sqrt{(2n+4)t} -d_f , \sqrt{(2n+4)t}  \) \cup CS(f) .
\]

(1) By Proposition \ref{prop_positivity_crit}, for each $(\xi ,\nu ) \in \rN CS (f)$ and any $x \in A^\circ_f ( 0, \sqrt{(2n+4)t} -d_f ) \cap ( \Pos \{ \nu \} + \xi )$, we have $u(x,t) > u(\xi ,t)$. Since $CS(f)$ is compact, there is a point $p \in CS(f)$ such that 
\[
u(p,t) = \min_{q \in CS (f)} u(q,t) .
\]
Thus, the restriction of $u(\cdot ,t)$ to the interior of $CS(f)+ (\sqrt{(2n+4)t} -d_f ) B^n$ is attained its minimum value at $p$.

(2) By Proposition \ref{prop_concavity_crit}, if $t$ is large enough, then, for each $(\xi , \nu) \in \rN CS(f)$, we have a unique point $x \in A_f ( \sqrt{(2n+4)t} -d_f , \sqrt{(2n+4)t} ) \cap ( \Pos \{ \nu \} + \xi )$ with $\nu \cdot \nabla u(x,t)=0$ and $\( \nu \cdot \nabla \)^2 u (x,t) <0$.
By the correspondence in Remark \ref{correspondence}, the implicit function theorem implies that the set 
\[
\Ga (t) := \bigcup_{( \xi , \nu ) \in \rN CS(f)} \left. \left\{ x \in A^\circ_f \( \sqrt{(2n+4)t} -d_f , \sqrt{(2n+4)t} \) \cap \( \Pos \{ \nu \} + \xi \) \rvert \nu \cdot \nabla u(x,t) =0 \right\} 
\]
is a smooth hyper-surface homeomorphic to $\pd CS(f)$. 
Since $\pd  CS(f)$ is homeomorphic to $S^{n-1}$, so is $\Ga (t)$.
The compactness of $\Ga (t)$ guarantees that the restriction of $u(\cdot ,t)$ to $\Ga (t)$ has maximum and minimum points which are critical points of $u(\cdot ,t)$.
\end{proof}

\begin{thm}
\label{crit2}
Let $f$ be as in Notation \ref{notation_main}.
We have
\[
\sup \left\{ \lvert x -m_f \rvert \lvert x \in \mathcal{C} \( u(\cdot ,t ) \) \cap \( CS(f) + \( \sqrt{(2n+4)t} -d_f \) B^n \)  \right\} \right.
=O \( \frac{1}{t} \)
\]
as $t$ goes to infinity.
\end{thm}

\begin{proof}
We give a proof for the even dimensional case. The argument works for the odd dimensional case.

Fix an arbitrary $x \in \mathcal{C} ( u(\cdot ,t ) ) \cap ( CS(f) + ( \sqrt{(2n+4)t} -d_f ) B^n )$. 
From Theorem \ref{crit1}, there exists a constant $T \geq d_f^2 /(2n+4)$ such that if $t \geq T$, then we have
\[
\mathcal{C} \( u(\cdot ,t ) \) \cap \( CS(f) + \( \sqrt{(2n+4)t} -d_f \) B^n \) \subset CS(f) .
\]
Thus, we may assume $x \in \mathcal{C} ( u(\cdot ,t ) ) \cap CS(f)$. 
Also, we take $T\geq d_f$, which implies that if $t \geq T$, then the ball $B^n_t(x)$ contains $CS(f)$.

By Lemma \ref{derivative}, we have
\[
- e^{-t/2} \( \int_{B^n_t(x)} \tilde{k}_{\frac{n+2}{2}} (r,t) f(y)dy \) x 
= - e^{-t/2} \( \int_{B^n_t(x)} \tilde{k}_{\frac{n+2}{2}} (r,t) f(y) y dy \)
-\frac{16}{\ga_n} \mathcal{E} \( \nabla u(x,t) \) .
\]
Also, we remark
\[
\( \int_{\supp f} f(y) dy \) m_f = \int_{\supp f} f(y) y dy .
\]
Hence there exists a positive constant $C$ such that if $t$ is large enough, then we have
\begin{align*}
&\frac{t^{n/2+2}}{(n+2)2^{n/2+1}} e^{-t/2} \( -\tilde{k}_{\frac{n+2}{2}} \( d_f ,t\) \) \| f \|_1 \lvert x- m_f \rvert  \\
&\leq \frac{t^{n/2+2}}{(n+2)2^{n/2+1}} e^{-t/2} \( \int_{\supp f} \( -\tilde{k}_{\frac{n+2}{2}} \( r ,t\) \)  f(y) dy \) \lvert x- m_f \rvert  \\
&\leq \frac{t^{n/2+2}}{(n+2)2^{n/2+1}} e^{-t/2} \lvert \int_{\supp f} \tilde{k}_{\frac{n+2}{2}} \( r ,t\)   f(y) \( m_f -y \) dy \rvert + \frac{16 t^{n/2+2}}{(n+2) \ga_n 2^{n/2+1}} \lvert \mathcal{E} \( \nabla u(x,t) \) \rvert \\
&= \lvert \int_{\supp f} \( 1+ O \( \frac{1}{t} \) \) f(y) \( m_f -y \) dy \rvert 
+\frac{16 t^{n/2+2}}{(n+2) \ga_n 2^{n/2+1}} \lvert \mathcal{E} \( \nabla u(x,t) \) \rvert \\
&\leq \frac{C d_f \| f \|_1}{t} +\frac{16 C(n)}{(n+2) \ga_n 2^{n/2+1}} e^{-t/2} t^{n/2+2} (1+t)^{n+1} \| f \|_{W^{(n+2)/2, \infty}} .
\end{align*}
Here, the first inequality follows from the second assertion in Lemma \ref{expansion_d}. 
The second inequality follows from the above remark on $x$ and the triangle inequality. 
The equation follows from the first assertion in Lemma \ref{expansion_d}. 
The last inequality follows from the above remark on $m_f$ and Lemma \ref{error}. 

By Lemma \ref{expansion_d}, if $t$ is large enough, then we have
\[
\frac{t^{n/2+2}}{(n+2)2^{n/2+1}} e^{-t/2} \( -\tilde{k}_{\frac{n+2}{2}} \( d_f ,t\) \) >\frac{1}{2} .
\]
Hence there exists a positive constant $C$ such that if $t$ is large enough, then we have
\[
\lvert x-m_f \rvert 
\leq \frac{C}{t} 
\]
for any $x \in \mathcal{C} ( u(\cdot ,t) ) \cap ( CS(f) +(\sqrt{(2n+4)t}-d_f )B^n )$.
\end{proof}

\section{Time-delayed hot and cold spots}

Let $f$ be as in Notation \ref{notation_main}, and $u$ the unique classical solution of \eqref{DWfg} with $f+g=0$. 
In this section, we discuss the large time behavior of the spatial maximum and minimum sets of $u$,
\begin{equation}
\mathcal{M} \( u(\cdot ,t ) \) 
= \left. \left\{ x \in \R^n \rvert u (x,t)  = \max  u ( \cdot ,t ) \right\}  ,\ 
\mathcal{M} \( -u(\cdot ,t ) \)  
= \left. \left\{ x \in \R^n \rvert u (x,t) = \min  u ( \cdot ,t ) \right\} .
\end{equation}

\begin{prop}
\label{lb_CS}
Let $f$ be as in Notation \ref{notation_main}.
There exists a constant $T>0$ such that if $t \geq T$, then, for any $x \in CS(f)$, we have
\[
-u(x,t) > 
\begin{cases}
\ds \frac{n \ga_n 2^{(n+1)/2}}{5 \sqrt{\pi} t^{n/2+1}} \| f \|_1  &( n \in 2 \N -1 ), \\
\ds \frac{n \ga_n 2^{n/2}}{5 t^{n/2+1}} \| f \|_1 &(n\in 2\N ).
\end{cases}
\]
\end{prop}

\begin{proof}
We give a proof for the even dimensional case. The argument works for the odd dimensional case. 

We take $T \geq d_f$, which implies that if $t \geq T$, then, for any $x \in CS(f)$, the ball $B^n_t(x)$ contains $CS(f)$.

By Proposition \ref{solution} and Lemma \ref{expansion_d}, if $t$ is large enough, then we have
\begin{align*}
\mathcal{P} \( u(x,t) \) 
&= -\frac{n\ga_n 2^{n/2}}{4t^{n/2+1}} \int_{B^n_t(x)} \exp \( -\frac{r^2}{2\( t +\sqrt{t^2 -r^2}\)} \) \( 1+ O \( \frac{1}{t} \) \) f(y) dy \\
&<-\frac{n \ga_n 2^{n/2}}{4t^{n/2+1}} \cdot \frac{19}{20} \int_{B^n_t(x)} \exp \( -\frac{r^2}{2\( t +\sqrt{t^2 -r^2}\)} \) f(y) dy \\
&<-\frac{n \ga_n 2^{n/2}}{4t^{n/2+1}} \cdot \frac{19}{20} \exp \( -\frac{d_f^2}{2 \( t +\sqrt{t -d_f^2} \)} \) \| f \|_1 \\
&< -\frac{n \ga_n 2^{n/2}}{4t^{n/2+1}} \cdot \frac{19}{20}  \cdot \frac{8}{9} \| f \|_1 
\end{align*}
for any $x \in CS(f)$. 
Here, the second inequality follows from $r \leq d_f$. 

By Lemma \ref{error}, if $t$ is large enough, then we have
\[
\lvert \mathcal{E} \( u(x,t) \) \rvert
\leq C(n) e^{-t/2} (1+t)^n \| f \|_{W^{n/2,\infty}} 
\leq \frac{n \ga_n 2^{n/2}}{4t^{n/2+1}} \cdot \frac{19}{20}  \cdot \frac{8}{9} \cdot \frac{1}{19} \| f \|_1 
\]
for any $x \in CS(f)$.

Hence if $t$ is large enough, then we have
\[
-u(x,t) 
> \frac{n \ga_n 2^{n/2}}{4t^{n/2+1}} \frac{19 \cdot 8 - 8}{20 \cdot 9} \| f \|_1 
=\frac{n \ga_n 2^{n/2}}{5 t^{n/2+1}} \| f \|_1
\]
for any $x \in CS(f)$.
\end{proof}

\begin{prop}
\label{ub_A}
Let $f$ be as in Notation \ref{notation_main}.
There exists a constant $T \geq d_f^2 /(2n+4)$ such that if $t \geq T$, then, for any $x \in A_f ( \sqrt{(2n+4)t}-d_f, \sqrt{(2n+4)t} )$, we have
\[
u (x,t)
< 
\begin{cases}
\ds \frac{7 \ga_n 2^{(n+1)/2}}{10 \sqrt{\pi} t^{n/2+1}} \exp \( -\frac{n+2}{2} \) \| f \|_1  &( n \in 2 \N-1 ),\\
\ds \frac{7 \ga_n 2^{n/2}}{10t^{n/2+1}} \exp \( -\frac{n+2}{2} \) \|  f \|_1 &( n \in 2 \N ).
\end{cases}
\]
\end{prop}

\begin{proof}
We give a proof for the even dimensional case. The argument works for the odd dimensional case. 

We take a large enough $T$ such that if $t \geq T$, then we have $t \geq \sqrt{(2n+4)t} +d_f$, which implies that, for any $x \in CS(f) +  \sqrt{(2n+4)t} B^n$, the ball $B^n_t(x)$ contains $CS(f)$.

By Proposition \ref{solution} and Lemma \ref{expansion_sqrt}, if $t$ is large enough, then we have
\begin{align*}
\mathcal{P} \( u(x,t) \) 
&= \frac{\ga_n 2^{n/2}}{4t^{n/2}} \int_{B^n_t(x)} \exp \( -\frac{r^2}{2\( t^2 +\sqrt{t^2 -r^2}\)} \) \( -\frac{n}{t} + \frac{1}{2} \( \frac{r}{t} \)^2 + O \( \frac{1}{t^2} \)  \)  f(y) dy \\
&<\frac{5\ga_n 2^{n/2}}{8t^{n/2+1}} \int_{B^n_t(x)} \exp \( -\frac{r^2}{2\( t^2 +\sqrt{t^2 -r^2}\)} \) \( 1+ O \( \frac{1}{t} \)  \)  f(y) dy \\
&<\frac{5 \ga_n 2^{n/2}}{8t^{n/2+1}} \cdot \frac{52}{51} \int_{B^n_t(x)} \exp \( -\frac{r^2}{2\( t^2 +\sqrt{t^2 -r^2}\)} \)   f(y) dy\\
&<\frac{5 \ga_n 2^{n/2}}{8t^{n/2+1}} \cdot \frac{52}{51} \exp \( -\frac{\( \sqrt{(2n+4)t} -d_f \)^2}{4t} \)   \| f \|_1 \\
&<\frac{5 \ga_n 2^{n/2}}{8 t^{n/2+1}} \cdot \frac{52}{51} \cdot \frac{14}{13} \exp \( -\frac{n+2}{2} \) \| f \|_1
\end{align*}
for any $x \in A_f ( \sqrt{(2n+4)t}-d_f, \sqrt{(2n+4)t} )$. 
Here, for the first inequality, we took a large enough $T>0$ such that if $t \geq T$, then we have
\[
r \leq \sqrt{(2n+4)t} + d_f \leq \sqrt{(2n+5)t} .
\]
The third inequality follows from $r \geq \sqrt{(2n+4)t} -d_f$. 

By Lemma \ref{error}, if $t$ is large enough, then we have
\[
\lvert \mathcal{E} \( u(x,t) \) \rvert
\leq C(n) e^{-t/2} (1+t)^n \| f \|_{W^{n/2,\infty}} 
\leq \frac{5 \ga_n 2^{n/2}}{8 t^{n/2+1}} \cdot \frac{52}{51} \cdot \frac{14}{13} \cdot \frac{1}{50} \exp \( -\frac{n+2}{2} \) \| f \|_1
\]
for any $x \in A_f ( \sqrt{(2n+4)t}-d_f, \sqrt{(2n+4)t} )$.

Hence if $t$ is large enough, then we have
\[
u(x,t) < \frac{5 \ga_n 2^{n/2}}{8 t^{n/2+1}} \cdot \frac{52}{51} \cdot \frac{14}{13} \cdot \frac{50+ 1}{50} \exp \( -\frac{n+2}{2} \) \| f \|_1
=\frac{7 \ga_n 2^{n/2}}{10 t^{n/2+1}} \exp \( -\frac{n+2}{2} \) \| f \|_1
\]
for any $x \in A_f ( \sqrt{(2n+4)t}-d_f, \sqrt{(2n+4)t} )$.
\end{proof}

\begin{prop}
\label{lb_A}
Let $f$ be as in Notation \ref{notation_main}.
There exists a constant $T \geq d_f^2 /(2n+4)$ such that if $t \geq T$, then, for any $x \in A_f ( \sqrt{(2n+4)t}-d_f, \sqrt{(2n+4)t} )$, we have
\[
u(x,t)
>
\begin{cases}
\ds \frac{3\ga_n 2^{(n+1)/2}}{32\sqrt{\pi} t^{n/2+1}} \exp \( -\frac{2n+5}{2} \) \| f \|_1  &( n \in 2 \N-1 ),\\
\ds \frac{3\ga_n 2^{n/2}}{32t^{n/2+1}} \exp \( -\frac{2n+5}{2} \) \|  f \|_1 &( n \in 2 \N ).
\end{cases}
\]
\end{prop}

\begin{proof}
We give a proof for the even dimensional case. The argument works for the odd dimensional case. 

We take a large enough $T$ such that if $t \geq T$, then we have $t \geq \sqrt{(2n+4)t} +d_f$, which implies that, for any $x \in CS(f) +  \sqrt{(2n+4)t} B^n$, the ball $B^n_t(x)$ contains $CS(f)$.

By Proposition \ref{solution} and Lemma \ref{expansion_sqrt}, if $t$ is large enough, then we have
\begin{align*}
\mathcal{P} \( u(x,t) \) 
&= \frac{\ga_n 2^{n/2}}{4t^{n/2}} \int_{B^n_t(x)} \exp \( -\frac{r^2}{2\( t^2 +\sqrt{t^2 -r^2}\)} \) \( -\frac{n}{t} + \frac{1}{2} \( \frac{r}{t} \)^2 + O \( \frac{1}{t^2} \)  \)  f(y) dy \\
&> \frac{3\ga_n 2^{n/2}}{8t^{n/2+1}} \int_{B^n_t(x)} \exp \( -\frac{r^2}{2\( t^2 +\sqrt{t^2 -r^2}\)} \) \( 1+ O \( \frac{1}{t} \) \) f(y) dy \\
&> \frac{3\ga_n 2^{n/2}}{16 t^{n/2+1}} \int_{B^n_t(x)} \exp \( -\frac{r^2}{2\( t^2 +\sqrt{t^2 -r^2}\)} \)  f(y) dy \\
&> \frac{3\ga_n 2^{n/2}}{16 t^{n/2+1}} \exp \( -\frac{2n+5}{2} \) \| f \|_1
\end{align*}
for any $x \in A_f ( \sqrt{(2n+4)t}-d_f, \sqrt{(2n+4)t} )$. 
Here, for the first and last inequalities, we took a large enough $T>0$ such that if $t \geq T$, then we have
\[
\sqrt{(2n+3)t} \leq \sqrt{(2n+4)t} - d_f \leq r \leq \sqrt{(2n+4)t} + d_f \geq \sqrt{(2n+5)t} .
\]
The second inequality is usual.

By Lemma \ref{error}, if $t$ is large enough, then we have
\[
\lvert \mathcal{E} \( u(x,t) \) \rvert
\leq C(n) e^{-t/2} (1+t)^n \| f \|_{W^{n/2,\infty}} 
\leq \frac{3\ga_n 2^{n/2}}{32 t^{n/2+1}} \exp \( -\frac{2n+5}{2} \) \| f \|_1
\]
for any $x \in A_f ( \sqrt{(2n+4)t}-d_f, \sqrt{(2n+4)t} )$.

Hence if $t$ is large enough, then we have
\[
u(x,t) < \frac{3\ga_n 2^{n/2}}{32 t^{n/2+1}} \exp \( -\frac{2n+5}{2} \) \| f \|_1
\]
for any $x \in A_f ( \sqrt{(2n+4)t}-d_f, \sqrt{(2n+4)t} )$.
\end{proof}

\begin{prop}
\label{ub_E}
Let $f$ be as in Notation \ref{notation_main}.
Let $\psi (t)$ be of small order of $t$ as $t$ goes to infinity with $0\leq \psi (t) \leq t$ for any $t>0$.
There exists a constant $T>0$ such that if $t \geq T$, then, for any $x \in (CS(f) + \psi (t) B^n)^c$, we have
\[
\lvert u(x,t) \rvert
< 
\begin{cases}
\ds \frac{3 \ga _n 2^{(n+1)/2}}{8 \sqrt{\pi} t^{n/2}} \exp \( - \frac{\psi (t)^2}{4t} \) \| f \|_1 &(n \in 2\N -1 ),\\
\ds \frac{3 \ga_n 2^{n/2}}{8t^{n/2}} \exp \( - \frac{\psi (t)^2}{4t} \) \| f \|_1 &(n \in 2\N  ) .
\end{cases}
\]
\end{prop}

\begin{proof}
We give a proof for the even dimensional case. The argument works for the odd dimensional case. 

By Remark \ref{expansion_k}, if $t$ is large enough, then we have
\begin{align*}
\lvert \tilde{k}_{\frac{n}{2}} (r,t) \rvert 
&\leq t k_{\frac{n+2}{2}} \( \frac{\sqrt{t^2-\psi (t)^2}}{2} \) + 2k_{\frac{n}{2}} \( \frac{\sqrt{t^2-\psi (t)^2}}{2} \)\\ 
&= \frac{2^{n/2}}{t^{n/2}} \( 1 -\( \frac{\psi (t)}{t} \)^2 \)^{-n/4} \exp \( \frac{\sqrt{t^2 -\psi (t)^2}}{2} \) \\
&\quad \times \left[ \( 1 -\( \frac{\psi (t)}{t} \)^2 \)^{-1/2} \( 1 -\frac{n(n+2)}{4t}\( 1 -\( \frac{\psi (t)}{t} \)^2 \)^{-1/2} + O\( \frac{1}{t^2} \) \) \right. \\
&\quad \quad \left. + \( 1-\frac{n(n-2)}{4t} \( 1 -\( \frac{\psi (t)}{t} \)^2 \)^{-1/2} + O\( \frac{1}{t^2} \) \) \right] \\
&< \frac{3 \cdot 2^{n/2}}{t^{n/2}} \exp \( \frac{\sqrt{t^2 -\psi (t)^2}}{2} \)
\end{align*}
for any $r \geq \psi (t)$. 
Here, the first inequality follows from the triangle inequality and the monotonicity of $k_\ell$. 
For the last inequality, we took a large enough $T$ such that if $t \geq T$, then we have
\begin{align*}
& \( 1 -\( \frac{\psi (t)}{t} \)^2 \)^{-n/4} \left[ \( 1 -\( \frac{\psi (t)}{t} \)^2 \)^{-1/2} \( 1 -\frac{n(n+2)}{4t}\( 1 -\( \frac{\psi (t)}{t} \)^2 \)^{-1/2} + O\( \frac{1}{t^2} \) \)  \right. \\
&\quad \left. + \( 1-\frac{n(n-2)}{4t} \( 1 -\( \frac{\psi (t)}{t} \)^2 \)^{-1/2} + O\( \frac{1}{t^2} \) \) \right] 
<3 .
\end{align*}
Therefore, if $t$ is large enough, then we have
\begin{align*}
\mathcal{P} \( u(x,t) \) 
&< \frac{3 \ga_n 2^{n/2}}{4 t^{n/2}} \exp \( -\frac{\psi (t)^2}{2\( t+ \sqrt{t^2 -\psi (t)^2} \)} \)  \| f \|_1 \\
&< \frac{3 \ga_n 2^{n/2}}{4 t^{n/2}} \exp \( -\frac{\psi (t)^2}{4t} \) \| f \|_1 
\end{align*}
for any $x \in (CS(f) + \psi (t) B^n)^c$.

By Lemma \ref{error}, if $t$ is large enough, then we have
\[
\lvert \mathcal{E} \( u(x,t) \) \rvert 
\leq C(n) e^{-t/2} \( 1+t \)^n \| f \|_{W^{n/2,\infty}} 
<  \frac{3 \ga_n 2^{n/2}}{8 t^{n/2}} \exp \( -\frac{\psi (t)^2}{4t} \) \| f \|_1 
\]
for any $x \in (CS(f) + \psi (t) B^n)^c$.

Hence if $t$ is large enough, then we have
\[
u(x,t) < \frac{3 \ga_n 2^{n/2}}{8 t^{n/2}} \exp \( -\frac{\psi (t)^2}{4t} \) \| f \|_1  .
\]
for any $x \in (CS(f) + \psi (t) B^n)^c$.
\end{proof}

\begin{thm}
\label{min}
Let $f$ be as in Notation \ref{notation_main}.
There exists a constant $T \geq d_f^2 /(2n+4)$ such that if $t \geq T$, then we have the following inclusion relations:
\begin{enumerate}[$(1)$]
\item $\mathcal{M} ( u(\cdot, t) ) \subset A^\circ_f ( \sqrt{(2n+4)t}-d_f ,\sqrt{(2n+4)t} )$.
\item $\mathcal{M} ( -u (\cdot ,t) )  = \mathcal{M} ( \vert u (\cdot ,t) \vert )  \subset CS(f)$.
\end{enumerate}
\end{thm}

\begin{proof}
Let $\psi (t)$ be of small order of $t$ as $t$ goes to infinity such that $1/ \psi (t)$ is of small order of $1/\sqrt{t}$ as $t$ goes to infinity, and $\sqrt{(2n+5)t} \leq \psi (t) \leq t$ for any $t>0$. We take a large enough $T$ such that if $t \geq T$, then we have $\sqrt{(2n+4)t} -d_f \geq \sqrt{2nt}$. 

(1) By the behavior of strictly increasing and decreasing of $u (\cdot ,t)$ in Propositions \ref{prop_positivity_crit} and \ref{prop_negativity_crit}, respectively, if $t$ is large enough, then there is no spatial maximum point of $u$ in $A_f (0, \sqrt{(2n+4)t}-d_f) \cup A_f (\sqrt{(2n+4)t} ,\psi (t))$. 

Propositions \ref{prop_negativity_null} and \ref{prop_positivity_null} guarantee the negativity and positivity of $u(\cdot ,t)$ on $CS(f)$ and $A_f (\sqrt{(2n+4)t}-d_f, \sqrt{(2n+4)t} )$, respectively. Thus, if $t$ is large enough, then there is no spatial maximum point of $u$ in $CS(f)$. 

By Propositions \ref{lb_A} and \ref{ub_E}, if $t$ is large enough, then there is no spatial maximum point of $u$ in the complement of $CS(f) +\psi (t) B^n$. Hence, we obtain the first assertion.

(2) By the behavior of strictly increasing and decreasing of $u(\cdot ,t)$ in Propositions \ref{prop_positivity_crit} and \ref{prop_negativity_crit}, respectively, if $t$ is large enough, then there is no spatial maximum point of $\vert u \vert$ in $A_f (0, \sqrt{(2n+4)t}-d_f) \cup A_f (\sqrt{(2n+4)t} ,\psi (t))$. 

By Propositions \ref{lb_CS} and \ref{ub_A}, if $t$ is large enough, then we have
\[
\min \left\{ -u(x,t)  \lvert x \in CS(f) \right\} \right. 
> \max \left\{ u(x,t) \lvert x \in A_f \( \sqrt{(2n+4)t}-d_f, \sqrt{(2n+4)t} \) \right\} \right. .
\]
Here, we remark $2e^{3/2} >7$. Propositions \ref{prop_negativity_null} and \ref{prop_positivity_null} guarantee the negativity and positivity of $u(\cdot ,t)$ on $CS(f)$ and $A_f (\sqrt{(2n+4)t}-d_f, \sqrt{(2n+4)t} )$, respectively. Thus, there is no spatial maximum point of $\vert u \vert$ in $A_f (\sqrt{(2n+4)t}-d_f, \sqrt{(2n+4)t} )$.

By Propositions \ref{lb_CS} and \ref{ub_E}, if $t$ is large enough, then there is no spatial maximum point of $\vert u \vert$ in the complement of $CS(f) + \psi (t) B^n$. Hence if $t$ is large enough, then all of spatial maximum points of $\lvert u \rvert$ are contained in $CS(f)$.

The negativity of $u$ on $CS(f)$ completes the proof of the second assertion.
\end{proof}

\begin{cor}
\label{centroid}
Let $f$ be as in Notation \ref{notation_main}.
We have
\[
\sup \left\{ \lvert x- m_f \rvert \lvert x \in \mathcal{M} \( - u(\cdot ,t)  \) \right\} \right. 
=O \( \frac{1}{t} \) 
\]
as $t$ goes to infinity.
\end{cor}

\begin{proof}
By Theorem \ref{min}, if $t$ is large enough, then we have
\[
\mathcal{M} \( -u (\cdot ,t) \) 
\subset \mathcal{C} \( u(\cdot ,t) \) \cap CS(f) .
\]
Hence Theorem \ref{crit2} implies the conclusion.
\end{proof}

\begin{prop}
\label{convex}
Let $f$ be as in Notation \ref{notation_main}.
There exists a constant $T >0$ such that if $t \geq T$, then the function $u(\cdot ,t)$ is strictly convex on $CS(f)$.
\end{prop}

\begin{proof}
We give a proof for the even dimensional case. The argument works for the odd dimensional case. 

We take $T \geq d_f$, which implies that if $t \geq T$, then, for any $x \in CS(f)$, the ball $B^n_t(x)$ contains $CS(f)$.

By Lemma \ref{expansion_d}, there exists a large enough $T$ such that if $t \geq T$, then $\tilde{k}_{\frac{n+4}{2}} (r,t)$ is negative for any $0\leq r \leq d_f$, and we have
\[
e^{-t/2} \( \tilde{k}_{\frac{n+4}{2}} (r,t) r^2 - 4 \tilde{k}_{\frac{n+2}{2}} (r,t) \) 
=\frac{(n+2)2^{n/2+3}}{t^{n/2+2}} \( 1+ O \( \frac{1}{t} \) \) 
> \frac{(n+2)2^{n/2+3}}{2 t^{n/2+2}} .
\]
Thus, by Lemma \ref{derivative}, if $t$ is large enough, then we have
\begin{align*}
\mathcal{P} \( \( \ome \cdot \nabla \)^2 u(x,t) \) 
&>\frac{\ga_n}{64} e^{-t/2} \int_{B^n_t(x)} \( \tilde{k}_{\frac{n+4}{2}} (r,t) r^2 - 4 \tilde{k}_{\frac{n+2}{2}} (r,t) \)  f(y) dy\\
&>\frac{(n+2) \ga_n 2^{n/2+3}}{128 t^{n/2+2}} \| f \|_1
\end{align*}
for any $\ome \in S^{n-1}$ and $x \in CS(f)$. 

By Lemma \ref{error}, if $t$ is large enough, then we have
\begin{align*}
\lvert \mathcal{E} \( \( \ome \cdot \nabla \)^2 u(x,t) \) \rvert
\leq C(n) e^{-t/2} \( 1+t \)^{n+2} \| f \|_{W^{(n+4)/2 ,\infty}} 
<\frac{(n+2) \ga_n 2^{n/2+3}}{256 t^{n/2+2}} \| f \|_1
\end{align*}
for any $\ome \in S^{n-1}$ and $x \in CS(f)$. 

Hence if $t$ is large enough, we have
\[
 \( \ome \cdot \nabla \)^2 u(x,t) 
>\frac{(n+2) \ga_n 2^{n/2+3}}{256 t^{n/2+2}} \| f \|_1
\]
for any $\ome \in S^{n-1}$ and $x \in CS(f)$. 
\end{proof}

Theorem \ref{min} and Proposition \ref{convex} lead to the uniqueness of a spatial minimum point of $u$.

\begin{cor}
\label{unique}
Let $f$ and $T$ be as in Theorem \ref{min}.
There exists a constant $T' \geq T$ such that if $t \geq T'$, then the function $u (\cdot ,t)$ has a unique minimum point.
\end{cor}

\section{Appendices}

\subsection{Frequently used Taylor's expansions}
Let us list up frequently used Taylor's expansions:
\begin{itemize}
\item As $s$ tends to zero, we have
\begin{equation}
\label{app_taylor}
(1-s^2)^\al = 1  - \al s^2 + \frac{\al (\al-1)}{2} s^4 + O \( s^6 \) .
\end{equation}
\item If a function $\psi (t)$ is of small order of $t$ as $t$ goes to infinity, then we have 
\begin{equation}
\label{app_taylor1}
	\( t^2-\psi (t)^2 \)^\al = t^{2\al}
		\( 1 -\al \( \frac{\psi (t)}{t} \)^2 + \frac{\al (\al-1)}{2} \( \frac{\psi (t)}{t} \)^4 
		+ O\( \( \frac{\psi (t)}{t} \)^6 \) \)
\end{equation}
as $t$ goes to infinity.
\item If a function $\f (t)$ is of small order of $\sqrt{t}$ as $t$ goes to infinity, then, as $t$ goes to infinity, we have the following expansions:
\begin{equation}
\label{app_taylor2}
\exp\( \frac{-t+\sqrt{t^2-\f (t)^2}}{2} \)
= 1 - \frac{\f (t)^2}{4t} + O\( \frac{\f (t)^4}{t^2}  \) ,\\
\end{equation}
\end{itemize}

\subsection{Properties of modified Bessel functions}
In this section, we collect some properties of the modified Bessel functions
\begin{equation}
\label{Bessel}
I_\ell (s)=\sum_{j=0}^\infty \frac{1}{j ! \Gamma (j+\ell +1)}
		\(\frac{s}{2}\)^{2j+\ell}
\end{equation}
used in this paper from [NO]:
\begin{itemize}
\item Direct computation shows the following recursion:
\begin{equation}
\label{Bessel_2}
	I_0' (s)=I_1(s),
	\ I_1'(s)=I_0(s)-\frac{1}{s}I_1(s),
	\ \frac{1}{s} \frac{d}{ds}\( \frac{I_\ell(s)}{s^\ell} \) =\frac{I_{\ell+1}(s)}{s^{\ell +1}}.
\end{equation}
\item The modified Bessel function $I_\nu (s)$ has the expansion
\begin{equation}
\label{Bessel_3}
\begin{split}
I_\ell (s)
=\frac{e^s}{\sqrt{2\pi s}}
	&\( 1-\frac{(\ell -1/2)(\ell +1/2)}{2s}
		+\frac{(\ell -3/2) (\ell -1/2)  (\ell +1/2) (\ell +3/2)}{8s^2} \right. \\
	&\quad \left. -\cdots+(-1)^i \frac{1}{i !2^i s^i}
		\prod_{j=1}^i \( \ell -(j-1/2) \) \( \ell +(j-1/2)\) 
	 +O\( \frac{1}{s^{i+1}} \) \)
\end{split}
\end{equation}
as $s$ goes to infinity.
\end{itemize}

\subsection{Geometric remarks}

Let $K$ be a convex body in $\R^n$, and
\begin{equation}
\rN K = \left\{ ( \xi , \nu ) \lvert \xi \in \pd K,\ \nu \in S^{n-1} ,\ \nu^\perp_- + \xi \supset K \right\} \right. .
\end{equation}
Fix two constants $\rho_2 > \rho_1 \geq 0$, and let $A_K (\rho_1 ,\rho_2)$ be the closed annulus
\begin{equation}
A_K \( \rho_1 ,\rho_2 \) = \left. \left\{ x \in \overline{\R^n \sm K} \rvert \rho_1 \leq \dist \( x , K\) \leq \rho_2 \right\}  .
\end{equation}
In this section, we show the geometric remarks needed in this paper.

\begin{rem}
\label{inscribed_ball}
Let $\Ome$ be a body (the closure of a bounded open set) in $\R^n$, $i_\Ome$ an incenter of $\Ome$, $\rho_\Ome$ the inradius of $\Ome$, and $K$ the convex hull of $\Ome$. We have
\[
B^n_{\rho_\Ome /2} \( i_\Ome \) \subset \bigcap_{(\xi ,\nu ) \in \rN K} \( \nu^\perp_- +  \xi -\frac{\rho_\Ome}{2} \nu  \) .
\] 
\end{rem}

\begin{proof}
Suppose that there is a point $p$ with  
\[
p \in B^n_{\rho_\Ome /2} \( i_\Ome \) \cap \( \bigcap_{(\xi ,\nu ) \in \rN K} \( \nu^\perp_- +  \xi -\frac{\rho_\Ome}{2} \nu  \) \)^c .
\]
There is a $( \xi , \nu ) \in \rN K$ with $(p - ( \xi -( \rho_\Ome /2) \nu ) ) \cdot \nu >0$.
Then, we have
\[
\frac{\rho_\Ome}{2} 
\leq \dist \( p , \pd \Ome \) 
\leq \dist \( p , \pd K \) 
< \frac{\rho_\Ome}{2} ,
\]
which is a contradiction.
\end{proof}

\begin{rem}\label{correspondence}
Let
\[
\Phi \( ( \xi, \nu ),\rho \) = \xi + \rho \nu ,\ \( (\xi ,\nu ) ,\rho \) \in \rN K \times \left[ \rho_1 ,\rho_2 \right] .
\]
\begin{enumerate}[$(1)$]
\item The image $\Phi ( \rN K \times [ \rho_1 ,\rho_2 ] )$ is contained in $A_K (\rho_1 ,\rho_2)$.
\item The correspondence $\Phi$ is a surjection from $\rN K \times [ \rho_1 , \rho_2 ]$ to $A_K (\rho_1 ,\rho_2)$.
\item The correspondence $\Phi$ is an injection from $\rN K \times [ \rho_1 , \rho_2 ]$ to $A_K (\rho_1 ,\rho_2)$.
\end{enumerate}
\end{rem}

\begin{proof}
(1) Take $( ( \xi , \nu ) , \rho)$ from $\rN K \times [ \rho_1 ,\rho_2 ]$. 
Since $\xi$ is a point on the boundary of $K$, we have
\[
\dist \( \xi +\rho \nu , \pd K \) \leq \lvert \xi + \rho \nu -\xi \rvert = \rho \leq \rho_2 .
\]
There is a point $\xi'$ on the boundary of $K$ such that $\vert \xi +\rho \nu -\xi' \vert = \dist (\xi +\rho \nu , \pd K )$. 
Then, the convexity of $K$ implies
\[
\lvert \xi +\rho \nu - \xi' \rvert^2 
= \lvert \xi -\xi' \rvert^2 +2 \rho \( \xi - \xi' \) \cdot \nu + \rho^2 
\geq \rho^2 \geq \rho_1^2 .
\]
Hence $\Phi$ is a map from $\mathrm{N} K \times [ \rho_1 ,\rho_2 ]$ to $A_K (\rho_1 ,\rho_2)$.

(2)  Fix a point $x$ from $A_K (\rho_1 ,\rho_2)$. There is a point $\xi$ on the boundary of $K$ such that $\vert x -\xi \vert = \dist (x, \pd K)$. Then, the convexity of $K$ guarantees $( \xi , ( x-\xi ) / \vert x- \xi \vert ) \in \rN K$, and we have
\[
\Phi \( \( \xi, \frac{x-\xi}{\lvert x-\xi \rvert} \) , \lvert x -\xi \rvert \) = \xi + \frac{x-\xi}{\lvert x-\xi \rvert} \lvert x -\xi \rvert \ =x .
\]
Hence $\Phi$ is a surjection from $\rN K \times [ \rho_1 ,\rho_2 ]$ to $A_K (\rho_1 ,\rho_2)$.

(3) Take $( ( \xi , \nu ), \rho)$ and $( ( \xi' , \nu' ) ,\rho')$ from $\rN K \times [ \rho_1 ,\rho_2 ]$. In order to show the injectivity of $\Phi$, we assume $\xi + \rho \nu = \xi' +\rho' \nu'$. If $\xi \neq \xi'$, then $K$ is contained in the intersection of two half spaces $\nu^\perp_- + ( \xi -\rho' \nu' )$ and $\nu'^\perp_- + ( \xi' -\rho \nu )$, which contradicts to the fact that $\xi$ and $\xi'$ are on the boundary of $K$. Therefore, we get $\xi = \xi'$. Then, Euclid's parallel postulate implies $\nu = \nu'$ and $\rho = \rho'$.
Hence $\Phi$ is an injection from $\rN K \times [ \rho_1 , \rho_2 ]$ to $A_K (\rho_1 ,\rho_2)$.
\end{proof}


\no Shigehiro SAKATA

\no E-mail: sakata@cc.miyazaki-u.ac.jp

\no Address: Faculty of Education, University of Miyazaki, 1-1 Gakuen Kibana-dai West, Miyazaki city, Miyazaki prefecture, 889-2192, Japan\\

\no Yuta WAKASUGI

\no E-mail: wakasugi.yuta.vi@ehime-u.ac.jp

\no Address: Graduate School of Science and Engineering,
Ehime University, 3, Bunkyo-cho, Matsuyama, Ehime, 790-8577, Japan

\end{document}